%% file: bl5vf.tex
\documentclass[reqno]{amsart}
\usepackage{amssymb}
\usepackage{amsmath}
\usepackage[mathscr]{euscript}

\usepackage{graphicx}
\usepackage{color}

\usepackage{epsfig}
\usepackage[small]{caption}
\usepackage{psfrag}

\makeatletter
\@addtoreset{equation}{section}
\makeatother

\renewcommand\thetable{\thesection.\@arabic\c@table}

\newtheorem{theorem}{Theorem}[section]
\newtheorem{lemma}[theorem]{Lemma}
\newtheorem{proposition}[theorem]{Proposition}
\newtheorem{corollary}[theorem]{Corollary}

\newtheorem{remark}[theorem]{Remark}

\newcommand{\mc}[1]{{\mathcal #1}}
\newcommand{\mf}[1]{{\mathfrak #1}}
\newcommand{\mb}[1]{{\mathbf #1}}
\newcommand{\bb}[1]{{\mathbb #1}}
\newcommand{\bs}[1]{{\boldsymbol #1}}
\newcommand{\ms}[1]{{\mathscr #1}}

\renewcommand{\Cap}{{\rm cap}}

\begin{document}

\title[Tunneling of the Kawasaki dynamics at low
temperatures]{Tunneling of the Kawasaki dynamics at low temperatures
  in two dimensions}

\author{J. Beltr\'an, C. Landim}

\address{\noindent IMCA, Calle los Bi\'ologos 245, Urb. San C\'esar
  Primera Etapa, Lima 12, Per\'u and PUCP, Av. Universitaria cdra. 18,
  San Miguel, Ap. 1761, Lima 100, Per\'u. 
\newline e-mail: \rm
  \texttt{johel.beltran@pucp.edu.pe} }

\address{\noindent IMPA, Estrada Dona Castorina 110, CEP 22460 Rio de
  Janeiro, Brasil and CNRS UMR 6085, Universit\'e de Rouen, Avenue de
  l'Universit\'e, BP.12, Technop\^ole du Madril\-let, F76801
  Saint-\'Etienne-du-Rouvray, France.  \newline e-mail: \rm
  \texttt{landim@impa.br} }

\keywords{Metastability, tunneling, lattice gases, Kawasaki dynamics,
  capacities}

\begin{abstract}
  Consider a lattice gas evolving according to the conservative
  Ka\-wa\-sa\-ki dynamics at inverse temperature $\beta$ on a two
  dimensional torus $\Lambda_L=\{0, \dots, L-1\}^2$ . We prove the
  tunneling behavior of the process among the states of minimal
  energy. More precisely, assume that there are $n^2$ particles, $n<
  L/2$, and that the initial state is the configuration in which all
  sites of the square $\{0, \dots, n-1\}^2$ are occupied. We
  show that in the time scale $e^{2\beta}$ the process evolves as a
  Markov process on $\Lambda_L$ which jumps from any site $\mb x$ to
  any other site $\mb y\not =\mb x$ at a strictly positive rate which
  can be expressed in terms of the hitting probabilities of simple
  Markovian dynamics.
\end{abstract}

\maketitle

\section{Introduction}
\label{sec0}

We introduced in \cite{bl2, bl7} a general method to describe the
asymptotic evolution of one-parameter families of continuous-time
Markov chains. This method has been succesfully applied in two
situations: For zero-range dynamics on a finite set which exhibit
condensation \cite{bl3, l1}, and for random walks evolving among
random traps \cite{jlt1, jlt2}. In the first model the chain admits a
finite number of ground sets, while in the second one there is a
countable number of ground states. We start in this paper the
investigation of a third case, where the limit dynamics is a
continuous process.

This article has two purposes. On the one hand, to derive some
estimates needed in the proof of the convergence, in the
zero-temperature limit, of the two-dimensional Kawasaki dynamics for
the Ising model in a large cube to a Brownian motion, presented in
\cite{gl5}. On the other hand, to illustrate the interest of the
method introduced in \cite{bl2, bl7} by applying it in a simple
context. A first step was done in this direction in \cite{bl4}, where
we derived the asymptotic behavior of continuous-time Markov chains
evolving on a \emph{fixed} and \emph{finite} state space imposing only
one simple condition on the jump rates. A second step is performed
here, applying the result obtained in \cite{bl4} to the Kawasaki
dynamics for the Ising model on a fixed two-dimensional square with
periodic boundary condition.

To present the main result of \cite{bl4}, consider a one-parameter
family of irreducible Markov chains $\eta_N(t)$ on a fixed and finite
state space $E$, reversible with respect to a probability measure
$\mu_N$. For example, the Glauber or the Kawaski dynamics for the
Ising model on a finite space. For $\eta\in E$, denote by $\mb P^N_\eta$ the
distribution of the process $\eta_N (t)$ starting from $\eta$. Expectation
with respect to $\mb P^N_\eta$ is represented by $\mb E^N_\eta$. 

Denote by $R_N(\eta,\xi)$ the jump rates of the chain and assume that for
all $\eta$, $\eta'$, $\xi$, $\xi'\in E$,
\begin{equation}
\label{c1}
\lim_{N\to\infty} \frac{R_N(\eta,\eta')}{R_N(\xi,\xi')} \;\in\; [0,\infty]
\end{equation}
in the sense that the limit exists with $+\infty$ as a possible value.
Note that conditions (2.1) and (2.2) in \cite{bl4} follow from
\eqref{c1}. Moreover, since for the Glauber or for the Kawasaki
dynamics the jump rates are either $1$ or $e^{-k\beta}$ for some $1\le
k\le 4$, assumption \eqref{c1} is fulfilled.

Under the elementary assumption \eqref{c1} we completely described in
\cite{bl4} the asymptotic evolution of the Markov chain
$\eta_N(t)$. More precisely, we proved the existence of a rooted tree
whose vertices are subsets of the state space. The tree fulfills the
following properties: (a) The root of the tree is the state space; (b)
the subsets of each generation form a partition of the state space;
and (c) the sucessors of a vertex are subsets of this vertex. To each
generation corresponds a tunneling behavior. Let $M+1$, $M\ge 1$, be
the number of generations of the tree, let $\kappa_m +1$ be the number of
descendents at generation $m+1$, $1\le m\le M$, and let $\mc E^m_1,
\dots, \mc E^m_{\kappa_m}, \Delta_m$ be the vertices of the generation
$m+1$.  We proved the existence of time scales $\theta^N_1 \gg \cdots
\gg \theta^N_M$ such that for each $1\le m\le M$:
\begin{enumerate}
\item For every $1\le i\le \kappa_m$ and every state $\eta$ in $\mc E^{m}_i$,
\begin{equation*}
\lim_{N\to \infty} \, \max_{\xi\in \mc E^{m}_i} \mb P^N_\xi \big[
H_{\breve{\mc E}^m_i} < H_{\eta} \big] \;=\; 0\;, 
\end{equation*}
where $\breve{\mc E}^m_i = \cup_{j\not =i} \mc E^m_j$ and where $H_A$
stands for the hitting time of a set $A\subset E$. This means that
starting from a set $\mc E^m_i$ the process visits all the points
of $\mc E^m_i$ before reaching another set $\mc E^m_j$.

\item Let $\mc E^{m} = \cup_i \mc E^m_i$ and let $\Psi_m : \mc E^{m}
  \to \{1, \dots, \kappa_m\}$ be the index function given by
\begin{equation*}
\Psi_m (\eta) \;=\; \sum_{i=1}^{\kappa_m} i\, \mb 1\{ \eta\in \mc E_i^{m} \}\; .
\end{equation*}
Denote by $\{\eta^{m}_N(t) : t\ge 0\}$ the trace of the process $\{\eta_N(t)
: t\ge 0\}$ on $\mc E^{m}$. For every $1\le i\le \kappa_m$, $\eta\in \mc
E_i^{m}$, under the measure $\mb P^N_\eta$, the (non-Markovian) index 
process $X^m_N(t) = \Psi_m (\eta^m_N(t \theta^N_m))$ converges to a
Markov process $X^m(t)$ on $\{1, \dots, \kappa_m\}$.

\item Starting from $\eta\in \mc E^m$, in the time scale $\theta^N_m$ the
  time spent outside $\mc E^m$ is negligible: For every $t>0$,
\begin{equation*}
\lim_{N\to \infty} \max_{\eta\in E} \, \mb E^N_\eta \Big[
\int_0^t \mb 1\{ \eta_N(s\theta^N_m) \in \Delta_m \} \, ds  \Big] \;=\; 0\;. 
\end{equation*}
\end{enumerate}

Therefore, in the time scale $\theta^N_m$ the process $\eta_N(t)$
behaves as a Markov process on a state space whose $\kappa_m$ points
are the sets $\mc E^m_1, \dots, \mc E^m_{\kappa_m}$ and which jumps
from $\mc E^m_i$ to $\mc E^m_j$ at a rate given by the jump rates of
the Markov process $X^m(t)$.

\medskip

We apply this result to investigate the zero-temperature limit of the
Kawasaki dynamics for the Ising model on a two-dimensional square with
periodic boundary condition. Here, for a fixed square and a fixed
number of particles, we derive the asymptotic behavior of the dynamics
among the ground states, configurations whose particles form
squares. In \cite{gl5}, we show that the evolution of these square
configurations converges to a Brownian motion when the lenght of the
square and the number of particles increase with the inverse of the
temperature.

\medskip 

The problem of describing the asymptotic behavior of a one-parameter
family of Markov chains evolving on a fixed and finite state space has
been considered before. Olivieri and Scoppola \cite{s1, os1} applied
the ideas introduced in the pathwise approach to
metastability \cite{cgov} to this context. They supposed that the jump
probabilities $P(x,y)$ of a discrete-time chain are given by
\begin{equation}
\label{c2}
P(x,y) = q(x,y) \, e^{-\beta [H(y)-H(x)]_+}\;,
\end{equation}
where $[a]_+$ represents the positive part of $a$, $q(x,y)$ a
symmetric function and $H$ an Hamiltonian. A subset $A$ of the state
space $E$ is called a cycle if $\max_{x\in A} H(x) <
\min_{y\in\partial_+ A} H(y)$, where $\partial_+ A$ stands for the
outer boundary of $A$: $\partial_+ A = \{y\not\in A : \exists\, x\in A
\,,\, P(x,y)>0\}$.  Under condition \eqref{c2}, Olivieri and Scoppola
proved that the exit time of a cycle, appropriately renormalized,
converges to an exponential random variable, and they obtained
estimates, with exponential errors, for the expectation of the exit
time. They were also able to describe the exit path from a
cycle. These results were generalized by Olivieri and Scoppola
\cite{os2} to the non reversible case, and by Manzo et
al. \cite{mnos1}, who extended the results proved in \cite{os1} for
the exit time of a cycle to the hitting time of the absolute minima of
the Hamiltonian.

Therefore, in the context of a fixed and finite state space, the
approach proposed in \cite{bl2, bl7} requires weaker assumptions on
the jump rates than the pathwise approach, it provides better
estimates on the exit times of the wells, and it characterizes the
transition probabilities which describe the way the process jumps from
one well to another. On the other hand, and in contrast with the
pathwise approach to metastability, it does not attempt to
characterize the exit paths from a well.

The potential theoretic approach to metastability \cite{begk1, begk2}
has also been tested \cite{bm1} in the framework of a fixed and finite
state space Markov chain.  Bovier and Manzo considered a Hamiltonian
$H$ and an irreducible discrete-time jump probability reversible with
respect to the Gibbs measure associated to the Hamiltonian $H$ at
inverse temperature $\beta$. Let $\mc M$ be a subset of the local
minima of the Hamiltonian $H$ and let $\mc M_x = \mc M\setminus
\{x\}$, $x\in\mc M$.  Under some assumptions on the Hamiltonian, they
computed the expectation of the hitting time of $\mc M_x$ starting
from $x\in\mc M$ and they proved that this hitting time properly
renormalized converges to an exponential random variable. They also
provided a formula for the probability that starting from $x\in\mc M$
the process returns to the set $\mc M$ at a local minima $y\in \mc M$
in terms of the right eigenvectors of the jump matrix of the
chain. This latter formula, although interesting from the theoretical
point of view, since it establishes a link between the spectral
properties of the generator and the metastable behavior of the
process, is of little pratical use because one is usually unable to
compute the eigenvectors of the generator.

In our approach, we replace the formula of the jump probabilities
written through eigenvectors of the generator by one, \cite[Remark 2.9
and Lemma 6.8]{bl2}, expressed only in terms of the capacities,
capacities which can be estimated using the Dirichlet and the Thomson
variational principles. This latter formula allows us to prove the
convergence of the process (in fact, of the trace process in the usual
Skorohod topology or of the original process in a weaker topology
introduced in \cite{jlt2}) by solving a martingale problem.

Metastability of locally conserved dynamics or of conservative
dynamics superposed with non-conservative ones have been considered
before. Peixoto \cite{p} examined the metastability of the two
dimensional Ising lattice gas at low temperature evolving according to
a superposition of the Glauber dynamics with a stirring dynamics.  Den
Hollander et al. \cite{hos1} and Gaudilli\`ere et al. \cite{gos}
described the critical droplet, the nucleation time and the typical
trajectory followed by the process during the transition from a
metastable set to the stable set in a two dimensional Ising lattice
gas evolving under the Kawasaki dynamics at very low temperature in a
finite square in which particles are created and destroyed at the
boundary. This result has been extended to the anistropic case by
Nardi et al. \cite{nos1} and to three dimensions by den Hollander et
al.  \cite{hnos1}. Using the potential theoretic approach introduced
in \cite{begk1, begk2}, Bovier et al. \cite{bhn1} presented the
detailed geometry of the set of critical droplets and provided sharp
estimates for the expectation of the nucleation time for this model in
dimension two and three.

More recently, Gaudilli\`ere et al. \cite{ghnos} proved that the
dynamics of particles evolving according to the Kawasaki dynamics at
very low temperature and very low density in a two-dimensional torus
whose length increases as the temperature decreases can be
approximated by the evolution of independent particles. These results
were used in \cite{bhs1}, together with the potential theoretic
approach alluded to above, to obtain sharp estimates for the
expectation of the nucleation time for this model.

\section{Notation and Results}
\label{sec1}

We consider a lattice gas on a torus subjected to a Kawasaki dynamics
at inverse temperature $\beta$.  Let $\Lambda_L = \{1, \dots, L\}^2$,
$L\ge 1$, be a square with periodic boundary conditions. Denote by
$\Lambda^*_L$ the set of edges of $\Lambda_L$. This is the set of
unordered pairs $\{x,y\}$ of $\Lambda_L$ such that $\Vert x-y\Vert =1$, where
$\Vert \,\cdot\,\Vert $ stands for the Euclidean distance. The configurations
are denoted by $\eta = \{\eta(x) : x\in \Lambda_L\}$, where $\eta(x)
=1$ if site $x$ is occupied and $\eta(x)=0$ if site $x$ is vacant.
The Hamiltonian $\bb H$, defined on the state space $\Omega_L =
\{0,1\}^{\Lambda_L}$, is given by
\begin{equation*}
-\; \bb H (\eta) \; =\; \sum_{\{x,y\}\in\Lambda^*_L} \eta(x) \eta(y)\; .
\end{equation*}
The Gibbs measure at inverse temperature $\beta$ associated to the
Hamiltonian $\bb H$, denoted by $\mu^\beta$, is given by
\begin{equation*}
\mu^\beta (\eta) \;=\; \frac 1{Z_\beta} e^{-\beta \bb H(\eta)}\;,
\end{equation*}
where $Z_\beta$ is the normalizing partition function.

We consider the continuous-time Markov chain $\{\eta^\beta_t : t\ge
0\}$ on $\Omega_L$ whose generator $L_\beta$ acts on functions
$f:\Omega_L \to \bb R$ as
\begin{equation*}
(L_\beta f)(\eta)\;=\; \sum_{\{x,y\}\in\Lambda^*_L} c_{x,y}(\eta) \, 
[f(\sigma^{x,y}\eta) - f(\eta)]\;,
\end{equation*}
where $\sigma^{x,y}\eta$ is the configuration obtained from $\eta$ by
exchanging the occupation variables $\eta(x)$ and $\eta(y)$:
\begin{equation*}
(\sigma^{x,y}\eta)(z) \;=\; 
\begin{cases}
\eta(z) & \text{if $z\not = x, y$}, \\
\eta(y) & \text{if $z = x$}, \\
\eta(x) & \text{if $z = y$}. \\
\end{cases}
\end{equation*}
The rates $c_{x,y}$ are given by
\begin{equation*}
c_{x,y} (\eta) \;=\; \exp\big\{-\beta \,
[\bb H(\sigma^{x,y}\eta) - \bb H(\eta)]_+ \big\}\;,
\end{equation*}
where $[a]_+$, $a\in \bb R$, stands for the positive part of $a$: $[a]_+
= \max\{a,0\}$. We sometimes represent $\eta^\beta_t$ by
$\eta^\beta(t)$ and we frequently omit the index $\beta$ of
$\eta^\beta_t$.

A simple computation shows that the Markov process $\{\eta_t : t\ge
0\}$ is reversible with respect to the Gibbs measures $\mu^\beta$,
$\beta>0$, and ergodic on each irreducible component formed by the
configurations with a fixed total number of particles. Denote by $|A|$
the cardinality of a finite set $A$. Let $\Omega_{L,K} = \{\eta\in
\Omega_L : \sum_{x\in \Lambda_L} \eta(x) = K\}$, $0\le K\le
|\Lambda_L|$, and denote by $\mu^\beta_K$ the Gibbs measure
$\mu^\beta$ conditioned on $\Omega_{L,K}$:
\begin{equation*}
\mu^\beta_K (\eta) \;=\; \frac 1{Z_{\beta, K}} e^{-\beta \bb
  H(\eta)}\;, \quad \eta \in \Omega_{L,K}\;,
\end{equation*}
where $Z_{\beta, K}$ is the normalizing constant $Z_{\beta, K} =
\sum_{\eta\in \Omega_{L,K}} \exp\{-\beta \bb H(\eta)\}$. We sometimes
denote $\mu^\beta_K$ simply by $\mu_K$.

For each configuration $\eta\in \Omega_{L,K}$, denote by $\mb
P^\beta_{\eta}$ the probability measure on the path space
$D([0,\infty), \Omega_{L,K})$ induced by the Markov process $\{\eta_t
: t\ge 0\}$ starting from $\eta$. Expectation with respect to $\mb
P^\beta_{\eta}$ is represented by $\mb E^\beta_{\eta}$.

Assume from now on that $K=n^2$ for some $4\le n < \sqrt{L}$, and
denote by $Q$ the square $\{0,\dots, n-1\}\times \{0,\dots, n-1\}$.
For $\mb x\in\Lambda_L$, let $Q_{\mb x} = \mb x + Q$ and let
$\eta^{\mb x}$ be the configuration in which all sites of the square
$Q_{\mb x}$ are occupied. Denote by $\Omega^0 = \Omega^0_{L,K}$ the
set of square configurations:
\begin{equation*}
\Omega^0 \;=\; \{\eta^{\mb x}: \mb x\in\Lambda_L \} \;.
\end{equation*}

If $L>2n$ the ground states of the energy $\bb H$ in $\Omega_{L,K}$
are the square configurations:
\begin{equation}
\label{b01}
\bb H_{\rm min}\;:=\;
\min_{\eta \in \Omega_{L,K}} \bb H(\eta) \;=\;  \bb H(\eta^{\mb
  x}) \;=\; - 2n(n-1)\;,
\end{equation}
and $\bb H(\eta) > -2n(n-1)$ for all $\eta\in\Omega_{L,K} \setminus
\Omega^0$.

To prove this claim, fix a configuration $\eta\in \Omega_{L,K}$.
Denote by $\xi$ the configuration obtained from $\eta$ by moving
vertically the particles of the configuration $\eta$ along the columns
of $\Lambda_L$ in the following way.  If there is a particle in the
column $C_k = \{\mb x = (x_1,x_2) \in\Lambda_L: x_1=k\}$, move a
particle in this column to the position $(k,0)$ if this site is
empty. Place all the other particles in the contiguous sites above
$(k,0)$. This means that if there are $j$ particles in the column
$C_k$ for the configuration $\eta$, $\xi(k,i)=1$ if and only if $0\le
i<j$.

This transformation does not decrease the number of vertical edges
connecting particles and maximizes the number of horizontal edges
among the configurations with a fixed number of particles per
column. Therefore, $\bb H(\xi) \le \bb H(\eta)$ and to prove claim
\eqref{b01} it is enough to show that $\bb H(\xi) \ge \bb H(\eta^{\mb
  x})$ and that the equality holds only if $\xi$ is a square
configuration. There are two cases which are examined
separately. Either all columns have at least one particle, or there is
a column with no particle.

In the second case, we may assume that $\xi$ is a configuration of
$\{0,1\}^{\bb Z^2}$ with $n^2$ occupied sites. If the set of occupied
sites is not a connected subgraph of $\bb Z^2$, we decrease the energy
by moving laterally a cluster of particles until it touches another
cluster. We may therefore suppose that the occupied sites form a
connected set.

Associate to each particle of $\xi$ a square of lenght $1$ centered at
the site occupied by the particle. Consider the smallest rectangle in
$\bb R^2$ which contains all squares. By construction, each
row and column of the rectangle contains at least one square.

Denote by $m_1\le m_2$ the lengths of the smallest rectangle which
contains all squares. The area of the rectangle, equal to $m_1m_2$,
must be larger than or equal to the number of particles $n^2$. It
follows from this inequality that $m_1+m_2\ge 2n$, with an equality if
and only if $m_1=m_2=n$. Since each row and each column contains at
least a square, there exist at least $2(m_1+m_2)$ edges connecting
an occupied site to an empty site.

Since there are $n^2$ particles, if all $4$ bonds of each particle
were attached to another particle, the energy would be $-2n^2$. For
the configurations $\xi$, we have seen that $2(m_1+m_2)$ bonds link a
particle to a hole. Hence, the energy of this configuration is at
least $-(2n^2 - m_1-m_2)\ge -2n(n-1)$, with an equality if and only if
$m_1=m_2=n$, i.e., if $\xi'$ is a square configuration. This proves
claim \eqref{b01} if the configuration $\xi$ can be considered as a
configuration of $\{0,1\}^{\bb Z^2}$.

Assume now that all columns have at least one particle. This means
that the configuration $\xi$ has a row of particles forming a ring
around the torus $\Lambda_L$. The argument presented below to estimate
the energy of $\xi$ applies also in the case where a column has no
particles. Let $h$ be the maximal height of the columns: $h=\max\{j\ge
1 : \exists\, k\,,\, \xi(k,j-1)= 1\}$. If all particles at height $h$
have two horizontal neighbors, the configuration $\xi$ forms a strip
around the torus $\Lambda_L$ with $hL=n^2$ particles and its energy is
equal to $-(2h-1)L = -(2n^2-L) > -2n(n-1)$ because $L>2n$ by
assumption.

If there is a particle with maximal height which has one or no
horizontal neighbor, we may move this particle to an empty site at
minimal height without increasing the energy. We repeat this operation
until reaching a configuration formed by a strip of particles
surmounted by a row of particles. Denote by $h$ the height of the
strip and by $0\le k<L$ the number of particles forming the top row so
that $n^2=hL+k$. The energy of this configuration is $-[2(hL+k)
-(L+1)] = -[2n^2 -(L+1)] > -2n(n-1)$ because $L>2n$ by
assumption. This concludes the proof of claim \eqref{b01}. \smallskip

We examine in this article the asymptotic evolution of the Markov
process $\{\eta_t : t\ge 0\}$ among the $|\Lambda_L|$ ground states
$\{\eta^{\mb x}: \mb x\in\Lambda_L \}$ in the zero temperature
limit. Denote by $\{\xi_t : t\ge 0\}$ the trace of the process
$\eta_t$ on the set of ground states $\Omega^0$. We refer to
\cite{bl2} for a precise definition of the trace process. The main
theorem of this article reads as follows.

\begin{theorem}
\label{s11} 
As $\beta\uparrow\infty$, the speeded up process $\xi (e^{2\beta} t)$
converges to a Markov process on $\Omega^0$ which jumps from
$\eta^{\mb x}$ to $\eta^{\mb y}$ at a strictly positive rate $r(\mb x,
\mb y)$. Moreover, in the time scale $e^{2\beta}$ the time spent by
the original process $\eta_t$ outside the set of ground states
$\Omega^0$ is negligible: for every $\mb x\in\Lambda_L$, $t>0$,
\begin{equation}
\label{18}
\lim_{\beta\to\infty} \mb E^\beta_{\eta^{\mb x}} \Big[ \int_0^t \mb
1\{\eta (e^{2\beta} s) \not \in \Omega^0\} \, ds\Big] \;=\; 0\;.
\end{equation}
\end{theorem}

In the terminology introduced in \cite{bl2}, the previous theorem
states that the sequence of Markov processes $\{\eta^\beta_t : t\ge
0\}$ exhibits a tunneling behavior on the time-scale $e^{2\beta}$,
with metastable sets $\{\{\eta^{\mb x}\} : \mb x\in \Lambda_L\}$,
metastable points $\eta^{\mb x}$ and asymptotic Markov dynamics
characterized by the strictly positive rates $r(\eta^{\mb x},
\eta^{\mb y})$.

\begin{remark}
\label{s21}
The asymptotic rates $r(\mb x, \mb y)$ depend on the parameters $L$
and $n$. We stress that these rates are strictly positive. The
asymptotic behavior is therefore non local, the limit process being
able to jump from a configuration $\eta^{\mb x}$ to any configuration
$\eta^{\mb y}$ with a positive probability.  We present in Corollary
\ref{s20} an explicit formula for these rates in terms of the hitting
probabilities of simple Markovian dynamics, and we examine in
\cite{gl5} the case in which $n$ and $L$ increase with $\beta$,
proving that the trace process $\xi_t$ converges in an appropriate
time scale to a two-dimensional Brownian motion.
\end{remark}

Denote by $\bb H_j$, $j\ge 0$, the set of configurations with energy
equal to $\bb H_{\rm min} + j = -2n(n-1)+j$:
\begin{equation*}
\bb H_j \;=\; \{\eta\in \Omega_{L,K} : \bb H(\eta) =
\bb H_{\rm min} +j\}\;, \quad 
\bb H_{ij} = \bigcup_{k=i}^j \bb H_k\;,
\end{equation*}
and let
\begin{equation}
\label{04}
\Delta_j \;=\; \{\eta\in\Omega_{L,K}: \bb H (\eta) >
\bb H_{\rm min} + j \} \;,
\end{equation}
so that $\bb H_0 = \Omega^0$ and $\{\bb H_{0j}, \Delta_j \}$ forms a
partition of the set $\Omega_{L,K}$.

\begin{remark}
\label{bs1}
The proof of Theorem \ref{s11} requires a precise knowledge of the
energy landscape of the Kawasaki dynamics in the graph
$\Omega_{L,K}$. This description is carried out in Section \ref{sec3},
where we show that the process $\eta_t$ visits solely a tiny portion
of the state space $\Omega_{L,K}$ during an excursion between two
ground states. More precisely, as illustrated in Figure \ref{fig1a},
we show the existence of four disjoint subsets of $\bb H_1$, denoted
by $\Omega^1, \dots, \Omega^4$, and of four subsets of $\bb H_2$,
denoted by $\Gamma_1, \dots, \Gamma_4$, such that, with a probability
converging to $1$ as the temperature vanishes,
\begin{enumerate}
\item After a time of order $e^{2\beta}$, the process jumps from a
  ground state to a configuration in the set $\Gamma_1$;
\item The process spends a time of order $1$ in a set $\Gamma_j$
  before reaching a configuration in $\Omega^{j-1}$ or in $\Omega^j$;
\item After a time of order $e^\beta$, the process jumps from a
  configuration in $\Omega^j$, $1\le j\le 4$, to a configuration in
  the set $\Gamma_j\cup\Gamma_{j+1}$;
\item The set $\Delta_2$ is never visited.
\end{enumerate}

A large portion of the set $\bb H_1$ can be reached from a ground
state only by crossing the set $\Delta_2$.  For example, the
configurations in which the particles form a $(n-k)\times (n+k)$
rectangle, $3\le k<\sqrt{n}$, with an extra row or column of particles
attached to the longest side of the rectangle.  By the previous
discussion, these configurations of the set $\bb H_1$ are never
visited during an excursion between two ground states.
\end{remark}

The simplicity of the energy landscape emerging from Remark \ref{bs1},
and illustrated in Figure \ref{fig1a}, is one of the main by-products of
this article. The proof of the convergence of the Kawasaki dynamics to
a Brownian motion in \cite{gl5} relies strongly on this description.

The article is organized as follows. In the next section we present a
sketch of the argument and in Section \ref{sec3} a description of the
shallow valleys visited during an excursion between two ground
states. In Section \ref{sec4}, we describe the evolution of the Markov
process among these shallow valleys in the time scale $e^\beta$, and
in Section \ref{sec2} the asymptotic behavior of the Kawasaki dynamics
among the ground states in the time scale $e^{2\beta}$.

{\tiny
\begin{figure}[htb]
  \centering
  \def\svgwidth{350pt}
  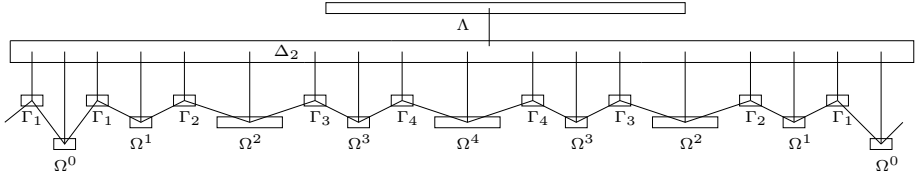
  \caption{The energy landscape of the Kawasaki dynamics at low
    temperature. $\Omega^0$ represents the set of ground states,
    $\Omega^j$, $1\le j\le 4$, disjoint subsets of $\bb H_1$,
    $\Gamma_j$, $1\le j\le 4$, disjoint subsets of $\bb H_2$, and
    $\Lambda= \bb H_1 \setminus [\cup_{1\le j\le 4} \Omega^j]$. The
    edges indicate that a configuration from one set may jump to the
    other. At low temperatures, during an excursion between two ground
    states the process does not visit the set $\Delta_2$ and all the
    analysis is reduced to the lower portion of the picture.}
\label{fig1a}
\end{figure}
}

\section{Sketch of the proof}

The proof of Theorem \ref{s11} relies on the strategy presented in
\cite{bl4} to prove the metastability of reversible Markov processes
evolving on finite state spaces. We do not investigate the full tree
structure of the chain, presented in the introduction, but a small
portion of it contained in the first and second generation of the tree.

A simple computation shows that assumptions (2.1) and (2.2) of that
article are satisfied. Indeed, since $\bb H(\sigma^{x,y}\eta) - \bb
H(\eta) = (\eta_y - \eta_x)\{\sum_{\Vert z-y\Vert =1} \eta_z -
\sum_{\Vert z-x\Vert =1}\eta_z + \eta_y - \eta_x\}$, the jump rates
$c_{x,y} (\eta)$ may only assume the values $1$, $e^{-\beta}$,
$e^{-2\beta}$ and $e^{-3\beta}$, which proves assumptions (2.1) and
(2.2).

Denote by $R_\beta(\eta,\xi)$ the rate at which the process
$\eta_t$ jumps from $\eta$ to $\xi$ so that $R_\beta(\eta,\xi) =
c_{x,y}(\eta)$ if $\xi=\sigma^{x,y}\eta$ for some bond $\{x,y\}\in
\Lambda_L^*$, and $R_\beta(\eta,\xi) = 0$, otherwise.

A self-avoiding path $\gamma$ from $\mc A$ to $\mc B$, $\mc A$, $\mc B\subset
\Omega_{L,K}$, $\mc A\cap \mc B = \varnothing$, is a sequence of
configurations $(\xi_0, \dots, \xi_n)$ such that $\xi_0\in \mc A$,
$\xi_n\in \mc B$, $\xi_i \not = \xi_j$, $i\not =j$, $R_\beta(\xi_j,
\xi_{j+1})>0$, $0\le j<n$. Denote by $\Gamma_{\mc A,\mc B}$ the set of
self-avoiding paths from $\mc A$ to $\mc B$ and let
\begin{equation*}
G_K (\mc A,\mc B) \;:=\; \max_{\gamma\in \Gamma_{\mc A,\mc B}} G_K(\gamma)\;, \quad
G_K(\gamma) \;=\;  G^\beta_K(\gamma) \;:=\; \min_{0\le i<n} \mu_K(\xi_i) R_\beta
(\xi_i,\xi_{i+1})
\end{equation*}
if $\gamma = (\xi_0, \dots, \xi_n)$. Since $\mu_K(\xi_i) R_\beta
(\xi_i,\xi_{i+1}) = \min\{\mu_K(\xi_i), \mu_K(\xi_{i+1})\}$,
$G_K(\gamma) = \min_{0\le i\le n} \mu_K(\xi_i)$ and
$G_K(\mc A,\mc B)$ is the measure of the saddle configuration from $\mc A$ to
$\mc B$. 

Denote by $D_K=D^\beta_K$ the Dirichlet form associated to the
generator of the Markov process $\eta_t$:
\begin{equation*}
D_K(f)\;=\; \frac 12 \, \sum_{\{x,y\}\in\Lambda^*_L} 
\sum_{\xi\in \Omega_{L,K}} \mu_K(\xi) \, 
c_{x,y}(\xi) \, \{f(\sigma^{x,y} \xi) - f(\xi)\}^2 \;, 
\quad f:\Omega_{L,K}\to \bb R\;. 
\end{equation*}
Let $\Cap_K(\mc A,\mc B) = \Cap^\beta_K(\mc A,\mc B)$, $\mc A$, $\mc
B\subset \Omega_{L,K}$, $\mc A\cap \mc B=\varnothing$, be the capacity
between $\mc A$ and $\mc B$:
\begin{equation*}
\Cap_K(\mc A,\mc B) \;=\; \inf_f D_K (f)\;,
\end{equation*}
where the infimum is carried over all functions $f: \Omega_{L,K}\to
\bb R$ such that $f(\xi) =1$ for all $\xi\in \mc A$, and $f(\xi) =0$ for
all $\xi\in \mc B$. We proved in \cite[Lemma 4.2 and 4.3]{bl4} that the
ratio $\Cap_K(\mc A,\mc B)/G_K(\mc A,\mc B)$ converges as
$\beta\uparrow\infty$: For every $\mc A$, $\mc B\subset \Omega_{L,K}$, $\mc A\cap
\mc B=\varnothing$,
\begin{equation}
\label{17}
\lim_{\beta\to \infty} \frac{\Cap_K(\mc A,\mc B)}{G_K(\mc A,\mc B)} \;=\;
C(\mc A,\mc B) \;\in\; (0,\infty)\;.
\end{equation}

We claim that $G_K(\{\eta^{\mb x}\}, \{\eta^{\mb y}\}) =
e^{-2\beta} \mu_K(\eta^{\mb x})$ for $\mb x \not =\mb y$.  Denote
by $e_1$, $e_2$ the canonical basis of $\bb R^2$.  On the one hand,
any path $\gamma$ from $\eta^{\mb x}$ to a set $\mc A\not\ni \eta^{\mb x}$
is such that $G_K(\gamma)\le e^{-2\beta} \mu_K(\eta^{\mb
  x})$. On the other hand, it is easy to construct a self-avoiding
path $\gamma =(\eta^{\mb x}=\xi_0, \dots, \xi_n=\eta^{\mb x+e_i})$
from $\eta^{\mb x}$ to $\eta^{\mb x+e_i}$, and therefore a path from
$\eta^{\mb x}$ to $\eta^{\mb y}$, such that $\mu_K(\xi_j)\ge
e^{-2\beta} \mu_K(\eta^{\mb x})$, $0\le j\le n$. This proves the
claim.

It follows from the previous claim and from \eqref{17} that
$\Cap_K(\{\eta^{\mb x}\}, \{\eta^{\mb y}\})$ is of order
$e^{-2\beta} \mu_K(\eta^{\mb x})$. In particular, to examine the
evolution of the process $\eta_t$ among the competing metastable
states $\eta^{\mb x}$ we need only to care of the states whose measure
are greater than or equal to $e^{-2\beta} \mu_K(\eta^{\mb
  x})$. Actually, as pointed out in Remark \ref{bs1}, only a much
smaller class is relevant for the problem.

In the next section we define the sets $\Omega^1, \dots, \Omega^4$
introduced in Remark \ref{bs1}. In the following section we show that
starting from a configuration in $\cup_{0\le j\le 4} \Omega_j$ in the
time scale $e^\beta$ the Kawasaki dynamics evolves as a markov chain
whose points are subsets of the sets $\Omega^j$. In this chain the
configurations $\eta^{\mb x}$ are absorbing points and the jump rates
are expressed as functions of the hitting probabilities of simple
Markovian dynamics.

In Section \ref{sec2}, we deduce from the previous result the
tunneling behavior of the process $\eta_t$ on the longer time scale
$e^{2\beta}$ among the competing metastable states $\eta^{\mb x}$. The
jump rates of this dynamics are expressed in terms of the hitting
probabilities of the absorbing states for the Markovian dynamics
derived in the previous step.

We conclude this section recalling the definition of a valley
presented in \cite{bl2}.  Denote by $H_\Pi$, $H^+_{\Pi}$, $\Pi\subset
\Omega_{L,K}$, the hitting time and the time of the first return to
$\Pi$:
\begin{equation*}
\begin{split}
& H_\Pi \;=\; \inf \big \{t>0 : \eta^\beta_t \in \Pi \big\}\;,
\\
& \quad H^+_\Pi \;=\; \inf \big \{t>0 : \eta^\beta_t \in \Pi 
\text{ and } \exists\; 0<s<t \,;\, \eta^\beta_s\not\in\Pi \big\}\;.    
\end{split}
\end{equation*}
We sometimes write $H(\Pi)$, $H^+(\Pi)$ instead of $H_\Pi$,
$H^+_\Pi$.

Consider two subsets $\mc W\subset \mc B$ of the state space $\Omega_K$
and a configuration $\eta\in\mc W$. The triple $(\mc W, \mc B,\eta)$
is called a valley if:
\begin{itemize}
\item Starting from any configuration of $\mc W$ the process visits
$\eta$ before hitting $\mc B^c$:
\begin{equation*}
\lim_{\beta\to\infty} \max_{\xi\in \mc W} \mb P^\beta_\xi [ H_{\mc B^c} < H_\eta ]
\;=\; 0\;.
\end{equation*}
\item There exists a sequence $m_\beta$ such that, for every $\xi\in\mc
  W$, $H_{\mc B^c}/m_\beta$ converges to a mean $1$ exponential random
  variable under $\mb P^\beta_\xi$. The sequence $m_\beta$ is called
  the depth of the valley.
\item The portion of time the process spends in $\mc B\setminus\mc W$
before hitting $\mc B^c$ is negligible: for every $\xi\in\mc W$ and
every $\delta>0$,
\begin{equation*}
\lim_{\beta\to\infty} \mb P^\beta_\xi \Big[  \frac 1{m_\beta} \int_0^{H_{\mc
    B^c}}  \mb 1\{ \eta_s \in \mc B \setminus \mc W\} > \delta \Big]
\;=\; 0\;.  
\end{equation*}
\end{itemize}

\section{Some shallow valleys}
\label{sec3}

We examine in this section the evolution of the Markov process
$\{\eta^\beta_t : t\ge 0\}$ between two consecutive visits to the
ground states $\{\eta^{\mb x} : \mb x\in\Lambda_L\}$. In the next
section, we show that at very low temperatures, in the time scale
$e^\beta$, much smaller than the time scale of an excursion between
ground states, the process $\eta_t$ evolves as a continuous-time
Markov chain whose state space consists of subsets of $\bb H_{01}$. In
this asymptotic dynamics the ground states play the role of absorbing
states. We present in this section the subsets of $\bb H_{01}$ which
become the points of the asymptotic dynamics and show that these sets
are valleys. We also provide explicit formulas for the jump rates of
the asymptotic dynamics in terms of four elementary Markov processes.

For a subset $B$ of $\Lambda_L$, denote by $\partial_+ B$ the
outer boundary of $B$. This is the set of sites which are at
distance one from $B$:
\begin{equation*}
\partial_+ B = \{\mb x \in\Lambda_L\setminus B : \exists
\, \mb y \in B \,,\, \Vert \mb y - \mb x\Vert =1 \} \;.
\end{equation*}

\subsection{Elementary Markov processes}
The jump rates among the shallow valleys introduced below are all
expressed in terms of the hitting probabilities of four elementary,
finite-state, continuous-time Markov chains. We present in this
subsection these processes and derive some identities needed later.
Let $\{\mb x_t : t\ge 0\}$ be the nearest-neighbor, symmetric random
walk on $\Lambda_L$ which jumps from a site $\mb x$ to $\mb x \pm e_i$
at rate $1$.  Denote by $\bb P^{\mb x}_{\mb y}$, $\mb y\in \Lambda_L$,
the probability measure on $D(\bb R_+, \Lambda_L)$ induced by $\mb
x_t$ starting from $\mb x$. We sometimes represent $\mb x_t$ by $\mb
x(t)$. Denote by $\mf p(\mb y, \mb z, G)$, $\mb y\in\Lambda_L$, $\mb
z\in G$, $G\subset\Lambda_L$, the probability that the random walk
starting from $\mb y$ reaches $G$ at $\mb z$:
\begin{equation*}
\mf p(\mb y, \mb z, G) \;:=\; \bb P^{\mb x}_{\mb y}
\big[\mb x (H_G) = \mb z \big]\; .
\end{equation*}
By extension, for a subset $A$ of $\Lambda_L$, let $\mf p(\mb x,
A, G) = \sum_{\mb y\in A} \mf p(\mb x, \mb y, G)$.  Moreover, when
$G=\partial_+Q$, we omit the set $G$ in the notation: $\mf p(\mb x,
A) \;:=\; \mf p(\mb x, A, \partial_+ Q)$. Let, finally,
\begin{equation}
\label{b09}
\mf p(A) \;:=\; \mf p(\mb w_2 + 2e_2 , A) 
\;+\; \mf p(\mb w_2 + e_1 + e_2 , A) \;.
\end{equation}
\smallskip

Let $\mb y_t = (\mb y^1_t, \mb y^2_t)$ be the continuous-time Markov
chain on $D_n = \{(j,k) : 0\le j < k\le n-1\} \cup \{(0,0)\}$ which
jumps from a site $\mb y$ to any of its nearest-neighbor sites $\mb
z$, $\Vert \mb y -\mb z\Vert =1$, at rate $1$. Let $D^+_n = \{(j,n-1)
: 0\le j<n-1\}$ and let
\begin{equation}
\label{b02}
\mf q_n = \bb P^{\mb y}_{(0,1)} \big[ H_{D^+_n} < H_{(0,0)}\big]\;,
\end{equation}
where $\bb P^{\mb y}_{(0,1)}$ stands for the distribution of $\mb y_t$
starting from $(0,1)$. \smallskip

Let $E_n = \{0,\dots, n-1\}^2$ and let $\mb z_t = (\mb z^1_t, \mb
z^2_t)$ be the continuous-time Markov chain on $E_n \cup \{\mf d\}$
which jumps from a site $\mb z\in E_n$ to any of its nearest-neighbor
sites $\mb z'\in E_n$, $\Vert \mb z' -\mb z\Vert =1$, at rate $1$, and
which jumps from $(1,1)$ (resp. from $\mf d$) to $\mf d$ (resp. to
$(1,1)$) at rate $1$. Let $E^+_n = \{(j,n-1) : 0\le j\le n-1\}$,
$E^-_n = \{(j,0) : 1\le j\le n-1\}$, $\partial E_n = E^+_n \cup E^-_n
\cup \{(0,0)\}$ and for $0\le k\le n-1$, let
\begin{equation}
\label{b03}
\begin{split}
& \mf r^+_n = \bb P^{\mb z}_{(0,1)} \big[ H_{\partial E_n} =
H_{E^+_n}\big]\;, \quad
\mf r^-_n = \bb P^{\mb z}_{(0,1)} \big[ H_{\partial E_n} =
H_{E^-_n}\big]\;, \\
&\qquad \mf r^0_n (k) = \bb P^{\mb z}_{(k,1)} \big[ H_{\partial E_n} =
H_{(0,0)}\big]\;, \quad \mf r_n = \mf r^+_n + \mf r^-_n\; ,
\end{split}
\end{equation}
where $\bb P^{\mb z}_{(k,1)}$ stands for the distribution of $\mb z_t$
starting from $(k,1)$. Note that the values of $\mf r^\pm_n$, $\mf
r^0_n (k)$ are unchanged if we consider the trace of $\mb z_t$ on the
set $\{0,\dots, n-1\}^2$. This latter process is a
nearest-neighbor random walk on $\{0,\dots, n-1\}^2$ whose
holding time at $(1,1)$ is longer. The embedded discrete-time chain of
the trace process is the symmetric, nearest-neighbor random walk. In
particular, $\mf r^+_n = (n-1)^{-1}$.

We claim that
\begin{equation}
\label{b10}
\sum_{k=1}^{n-1} \mf r^0_n (k) \;=\; \mf r^-_n \;.
\end{equation}
Indeed, denote by $\bb Z_k$ the embedded, discrete-time chain on
$E_n$. Outside of the boundary $\partial E_n$, $\bb Z_k$ jumps
uniformly to one of its neighbors. At the boundary $\partial E_n$, it
jumps with probability $1$ to the unique neighbor in the interior $E_n
\setminus \partial E_n$. Therefore,
\begin{equation*}
\mf r^0_n (k) \;=\; \bb P^{\mb z}_{(k,1)} \big[ H_{\partial E_n} =
H_{(0,0)}\big]\; =\; \sum_{\gamma} p(\gamma) \, 3^{-1}
\; =\; \sum_{\gamma} \pi(k,1)\, p(\gamma) \, \frac 1{3\pi(k,1)}\;,
\end{equation*}
where the sum is carried over all paths $\gamma$ from $(k,1)$ to
$(0,1)$ which never pass by $\partial E_n$. The factor $1/3$
represents the probability to jump from $(0,1)$ to $(0,0)$ and $\pi$
the reversible stationary measure for the chain $\bb Z_k$, which is
proportional to the degree of the vertices. By reversibility, the
previous sum is equal to
\begin{equation*}
\sum_{\gamma'} \pi(0,1)\, p(\gamma') \, \frac 1{3\pi(k,1)} 
\;=\; \sum_{\gamma'} p(\gamma') \, \frac 1{Z_n\, \pi(k,1)} \;,
\end{equation*}
where the sum is now carried over all paths $\gamma'$ from $(0,1)$ to
$(k,1)$ which never pass by $\partial E_n$ and $Z_n$ is the sum of the
degrees of all vertices. Last expression is equal to $\bb P^{\mb
  z}_{(0,1)} \big[ H_{\partial E_n} = H_{(k,0)}\big]$. Summing over
all $k$ yields \eqref{b10}. \smallskip

Finally, consider two independent, nearest-neighbor, continuous-time,
random walks $\mb u_t$, $\mb v_t$ evolving on an interval $J=\{m,
\dots, M\}$, $m<M$, which jump from $k$ (resp. $k+1$) to $k+1$
(resp. $k$), $m\le k<M$, at rate $1$.  Let $H_1 = \inf\{t>0: |\mb
u_t-\mb v_t|=1\}$. For $a,a+2\in J$, $b,b+1\in J$, let
\begin{equation}
\label{b13}
\mf m(J,a,b) \;:=\; \bb P^{\mb u\mb v}_{(a,a+2)} \big[ (\mb u_{H_1},\mb v_{H_1}) =
(b,b+1)\big]\; .
\end{equation}
where $\bb P^{\mb u\mb v}_{(a,a+2)}$ stands for the distribution of
the pair $(\mb u_t, \mb v_t)$ starting from $(a,a+2)$.

\subsection{The distribution of $\eta(H^+_{\bb H_{01}})$ starting from a
  ground state}
\label{ss5} 

Denote by $\mc N(\eta^{\mb x})$ the set of eight configurations which
can be obtained from $\eta^{\mb x}$ by a rate $e^{-2\beta}$ jump. Two
of these configurations deserve a special notation, $\eta^\star_1 =
\sigma^{\mb w_2, \mb w_2+e_2} \eta^{\mb w}$ and $\eta^\star_2 =
\sigma^{\mb w_2, \mb w_2+e_1} \eta^{\mb w}$.

\begin{lemma}
\label{bs01}
For each $\eta\in \mc N(\eta^{\mb x})$, there exists a probability
measure $\bb M(\eta , \cdot)$ on $\bb H_{01}$ such that
\begin{equation}
\label{b04}
\bb M(\eta, \mc A) \;:=\; \lim_{\beta\to\infty} \mb P_{\eta} 
[ H_{\bb H_{01}} = H_{\mc A} ]
\end{equation}
for all $\mc A \subset \bb H_{01}$. Set $\bb M_j (\cdot) = \bb
M(\eta^\star_j , \cdot)$, $j=1$, $2$. Then,
\begin{equation*}
\begin{split}
& \bb M_1(\mc E^{0,0}_{\mb w}) \;=\; \frac 12 \Big\{ \,
\frac {1 + \mf r^-_n + \mf A_{1,2}}
{4 + \mf q_n + \mf r_n - \mf A}
\;+\; \frac{1 + \mf r^-_n + \mf A_2 - \mf A_1}
{4 + \mf q_n + \mf r_n + \mf A(e_1) - \mf A(e_2)}
 \Big\} \; , \\
& \quad \bb M_2(\mc E^{0,0}_{\mb w}) \;=\; \frac 12 \Big\{ \,
\frac {1 + \mf r^-_n + \mf A_{1,2}}
{4 + \mf q_n + \mf r_n - \mf A}
\;-\;
\frac{1 + \mf r^-_n + \mf A_2 - \mf A_1}
{4 + \mf q_n + \mf r_n + \mf A(e_1) - \mf A(e_2)}
\Big\} \;, 
\end{split}
\end{equation*}
where $\mf A(e_i) = \mf p(\mb w_2 + e_i)$, $\mf A_j = \mf
p(Q^{2,j}_{\mb w})$ and
\begin{equation*}
\mf A \;=\; \; \mf A(e_1) \;+\; \mf A(e_2)\;,
\quad \mf A_{1,2} \;=\; \mf A_1 \;+\; \mf A_2 \;,
\quad \mf A_{0,3} \;=\; \mf A_0 \;+\; \mf A_3 \;.
\end{equation*}
\end{lemma}

\begin{proof}
We prove this lemma for $\eta=\eta^\star_1$, $\mc A = \mc E^{2,2}_{\mb
  0}$ and leave the other cases to the reader. Recall the definition
of $\mf q_n$, $\mf r^\pm_n$ introduced in \eqref{b02}, \eqref{b03}. We
claim that any limit point $\bb M_j = \bb M_j(\mc E^{2,2}_{\mb w})$ of the
sequences $\mb P_{\eta^\star_j} [ H_{\bb H_{01}} = H_{\mc A} ]$
satisfies the equations
\begin{equation}
\label{b05}
\begin{split}
& (4 + \mf q_n + \mf r_n - \mf A) \, (\bb M_1 + \bb M_2) \;=\;
1\;+\; \mf A_{1,2} + \mf r^-_n \;.
\end{split}
\end{equation}


To prove \eqref{b05}, assume that $\mb P_{\eta^\star_j} [ H_{\bb
  H_{01}} = H_{\mc A} ]$ converges and observe that the configuration
$\eta^\star_1$ may jump at rate $1$ to $6$ configurations and at rate
$e^{-\beta}$ or less to $O(n)$ configurations. Among the
configurations which can be reached at rate $1$ two belong to $\bb
H_{01}$, one of them being $\eta^{\mb w}$ and the other $\sigma^{\mb
  w_2, \mb w_2+e_2-e_1} \eta^{\mb w} \in\mc A$. Hence, if we denote by
$\mc N_2(\eta^\star_1)$ the set of the remaining four configurations
which can be reached from $\eta^\star_1$ by a rate $1$ jump,
decomposing $\mb P_{\eta^\star_1} \big[ H_{\bb H_{01}} = H_{\mc A}
\big]$ according to the first jump we obtain that
\begin{equation}
\label{b06}
6\, \mb P_{\eta^\star_1} \big[ H_{\bb H_{01}} = H_{\mc A} \big]\;=\;
1 \;+\; \sum_{\eta' \in \mc N_2(\eta^\star_1)}
\mb P_{\eta'} \big[ H_{\bb H_{01}} = H_{\mc A}
\big]  \;+\; \epsilon (\beta)\;,
\end{equation}
where $\epsilon (\beta)$ is a remainder which vanishes as
$\beta\uparrow\infty$. 

We examine the four configurations of $\mc N_2(\eta^\star_1)$
separately. In two configurations, $\sigma^{\mb w_2, \mb w_2+2e_2}
\eta^{\mb w}$ and $\sigma^{\mb w_2, \mb w_2+e_2+e_1} \eta^{\mb w}$, a
particle is detached from the quasi-square $Q^{2}_{\mb w}$.  The
detached particle performs a rate $1$ symmetric random walk on
$\Lambda_L$ until it reaches the outer boundary of the square $Q_{\mb
  0}$. Denote by $H_{\rm hit}$ the time the detached particle hits the
outer boundary of the square $Q_{\mb w}$. Among the remaining
particles, two jumps have rate $e^{-\beta}$ and the other ones have
rate at most $e^{-2\beta}$. Therefore, by the strong Markov property,
for $\eta' = \sigma^{\mb w_2, \mb w_2+2e_2} \eta^{\mb w}$,
$\sigma^{\mb w_2, \mb w_2+e_2+e_1} \eta^{\mb w}$,
\begin{equation*}
\begin{split}
\mb P_{\eta'} \big[ H_{\bb H_{01}} = H_{\mc A}
\big] \;& =\; \mb P_{\eta'} \big[ \eta(H_{\rm hit} ) \in  \mc
E^{2,2}_{\mb w} \big] \\
& +\; 
\sum_{i=1}^2 \mb P_{\eta'} \big[ \eta(H_{\rm hit} ) = \eta^\star_i \big] \, 
\mb P_{\eta^\star_i} \big[ H_{\bb H_{01}} = H_{\mc A}\big] 
\;+\; \epsilon(\beta) \;.
\end{split}
\end{equation*}
By definition \eqref{b09} of $\mf p$, the contribution of the terms
$\eta' = \sigma^{\mb w_2, \mb w_2+2e_2} \eta^{\mb w}$ and $\eta' =
\sigma^{\mb w_2, \mb w_2+e_2+e_1} \eta^{\mb w}$ to the sum appearing
on the right hand side of \eqref{b06} is
\begin{equation}
\label{b07}
\mf p(\mb w_2 + e_2) \,
\mb P_{\eta^\star_1} \big[ H_{\bb H_{01}} = H_{\mc A}\big] 
\;+\; \mf p(\mb w_2 + e_1) \,
\mb P_{\eta^\star_2} \big[ H_{\bb H_{01}} = H_{\mc A}\big]  \;+\; 
\mf p(Q^{2,2}_{\mb w}) \;+\; \epsilon (\beta) \;.
\end{equation} 

It remains to analyze the two configurations of $\mc N_2(\eta)$,
$\eta' = \sigma^{\mb w_2-e_2, \mb w_2+e_2} \eta^{\mb w}$ and $\eta' =
\sigma^{\mb w_2-e_1, \mb w_2+e_2} \eta^{\mb w}$.  In the first one, if
we denote by $\mb z^1_t$ the horizontal position of the particle
attached to the top side of the square $Q$ and by $\mb z^2_t$ the
vertical position of the hole on the left side of the square, it is
not difficult to check that $(\mb z^1_t, \mb z^2_t)$ evolves as the
Markov chain described just before \eqref{b03} with initial condition
$(\mb z^1_0, \mb z^2_0) = (0,1)$.

Denote by $H_{\rm hit}$ the time the hole hits $0$ or $n-1$.  Since
the the hole moves at rate $1$, and since all the other $O(n)$
possible jumps have rate at most $e^{-\beta}$, with probability
increasing to $1$ as $\beta\uparrow\infty$, $H_{\rm hit}$ occurs
before any rate $e^{-\beta}$ jump takes place.  At time $H_{\rm hit}$
three situations can happen. If the process $(\mb z^1_t, \mb z^2_t)$
reached $(0,0)$ (resp. $E^+_n$, $E^-_n$), the process $\eta(t)$
returned to the configuration $\eta^\star_1$ (resp. hitted a configuration
in $\mc E^{1,2}_{\mb w}$, $\mc E^{2,2}_{\mb w}$). Hence, by definition
of $\mf r^\pm_n$,
\begin{equation*}
\mb P_{\eta'} \big[ H_{\bb H_{01}} = H_{\mc A} \big] \;=\; 
\mf r^-_n \;+\; (1-\mf r_n)\,
\mb P_{\eta^\star_1} \big[ H_{\bb H_{01}} = H_{\mc A} \big] \;+\; \epsilon
(\beta) 
\end{equation*}
for $\eta' = \sigma^{\mb w_2-e_2, \mb w_2+e_2} \eta^{\mb w}$. 

Assume now that $\eta' = \sigma^{\mb w_2-e_1, \mb w_2+e_2} \eta^{\mb
  0}$.  In this case, if we denote by $\mb y^1$ the horizontal
position of the particle attached to the side of the quasi-square and
by $\mb y^2$ the horizontal position of the hole, the pair $(\mb y^1_t
, \mb y^2_t)$ evolve according to the Markov process introduced just
before \eqref{b02} with initial condition $(\mb y^1_0 , \mb y^2_0) =
(0 , 1)$.  Denote by $H_{\rm hit}$ the time the hole hits $0$ or
$n-1$. Here again, since the attached particle and the hole move at
rate $1$ and since all the other $O(n)$ jumps have rate at most
$e^{-\beta}$, with asymptotic probability equal to $1$, $H_{\rm hit}$
occurs before any rate $e^{-\beta}$ jump takes place. At time $H_{\rm
  hit}$, if $\mb y^2_{H_{\rm hit}}=0$, the process $\eta(t)$ has
returned to the configuration $\eta^\star_1$, while if $\mb y^2_{H_{\rm
    hit}}=n-1$, the process $\eta(t)$ has reached a configuration in
$\mc E^{3,2}_{\mb w}$. Since the the random walk reaches $n-1$ before
$0$ with probability $\mf q_n$, by the strong Markov property
\begin{equation*}
\mb P_{\eta'} \big[ H_{\bb H_{01}} = H_{\mc A} \big] \;=\; 
(1-\mf q_n)\,
\mb P_{\eta^\star_1} \big[ H_{\bb H_{01}} = H_{\mc A} \big] \;+\;\epsilon
(\beta) 
\end{equation*}
for $\eta' = \sigma^{\mb w_2-e_1, \mb w_2+e_2} \eta^{\mb w}$. 

Therefore, the contribution of the last two configurations of $\mc
N_2(\eta^\star_1)$ to the sum on the right hand side of \eqref{b06} is
\begin{equation}
\label{b08}
\mf r^-_n \;+\; 
\Big( 2 - \mf r_n - \mf q_n\Big) \, 
\mb P_{\eta^\star_1} \big[ H_{\bb H_{01}} = H_{\mc A} \big] \;+\; \epsilon
(\beta) \;.
\end{equation}

Equations \eqref{b06}, \eqref{b07} and \eqref{b08} yield a linear
equation for $\bb M_1=\mb P_{\eta^\star_1} [ H_{\bb H_{01}} = H_{\mc A} ]$
in terms of $\bb M_1$ and $\bb M_2=\mb P_{\eta^\star_2} [ H_{\bb H_{01}} =
H_{\mc A} ]$. Analogous arguments provide a similar equation for $\bb M_2$
in terms of $\bb M_1$ and $\bb M_2$. Adding these two equations we obtain
\eqref{b05}, while subtracting them gives a formula for the difference
$\bb M_1-\bb M_2$. The assertion of the lemma follows from these equations for
$\bb M_1+\bb M_2$ and $\bb M_1-\bb M_2$.
\end{proof}

Similar computations to the ones carried over in the previous proof
permit to derive explicit expressions for $\bb M_j$. For example, we
have that
\begin{equation*}
\bb M_1(\eta^{\mb w}) \;=\; \bb M_2(\eta^{\mb w}) \;=\; 
\frac {1} {4 + \mf q_n + \mf r_n - \mf A}\;\cdot
\end{equation*}
By symmetry, for each $\eta\in \mc N(\eta^{\mb x})$, we can represent
$\bb M(\eta,\cdot)$ in terms of $\bb M_1$ and $\bb M_2$.  Moreover,
for all $\eta\in \mc N(\eta^{\mb x})$,
\begin{equation*}
\bb M(\eta, \eta^{\mb x}) \;+\; \sum_{0\le i,j\le 3} \bb M(\eta, \mc
E^{i,j}_{\mb x}) \;=\; 1\; .
\end{equation*}

\subsection{The valleys $\mc E^{i,j}_{\mb x}$}
\label{ss1}

Let $Q^i \;=\; Q \setminus \{\mb w_i \}$, $0\le i\le 3$, where
\begin{equation*}
\mb w_0 = \mb w = (0,0)\;, \quad \mb w_1 =
(n-1,0)\; , \quad \mb w_{2} = (n-1,n-1)\;,  \quad \mb w_3 = (0,n-1) 
\end{equation*}
are the corners of the square $Q$. For $\mb x\in\Lambda_L$, let
$Q_{\mb x}^i = \mb x + Q^i$, $\mb x_i = \mb x + \mb w_i$.

Denote by $\partial_j Q_{\mb x}^i$, $0\le j\le 3$, the $j$-th boundary
of $Q_{\mb x}^i$:
\begin{equation*}
\begin{split}
& \partial_j Q_{\mb x}^{i} =\{\mb z\in \partial_+ Q_{\mb x}^i : 
\exists\, \mb y\in Q_{\mb x}^{i} \,;\, \mb y -\mb z = (1-j) e_2\} 
\quad j=0,2 \; , \\
& \quad \partial_j Q_{\mb x}^{i} =\{\mb z\in \partial_+ Q_{\mb x}^i : 
\exists\, \mb y\in Q_{\mb x}^{i} \,;\, \mb y -\mb z = (j-2) e_1\} 
\quad j=1,3\;.
\end{split}
\end{equation*}
Let $Q^{i,j}_{\mb x} = \partial_j Q_{\mb x}^{i} \setminus Q_{\mb x}$,
let $\mc E^{i,j}_{\mb x}$ be the set of configurations in which all
sites of the set $Q^{i}_{\mb x}$ are occupied with an extra particle
at some location of $Q^{i,j}_{\mb x}$, and let $\Omega^1 =
\Omega^1_{L,K}$ be the union of all such sets:
\begin{equation*}
\mc E^{i,j}_{\mb x} \;=\; \{\sigma^{\mb x_i, \mb z} \eta^{\mb x} : 
\mb z\in Q^{i,j}_{\mb x}\}\;, \quad
\Omega^1_{\mb x} \;=\; \bigcup_{0\le i,j\le 3} \mc E^{i,j}_{\mb x}\;, \quad 
\Omega^1 \;=\; \bigcup_{\mb x\in\Lambda_L} \Omega^1_{\mb x}\;.
\end{equation*}
Note that $\Omega_1\subset \bb H_1$. 

\begin{figure}[!h]
\begin{center}
\includegraphics[scale =0.5]{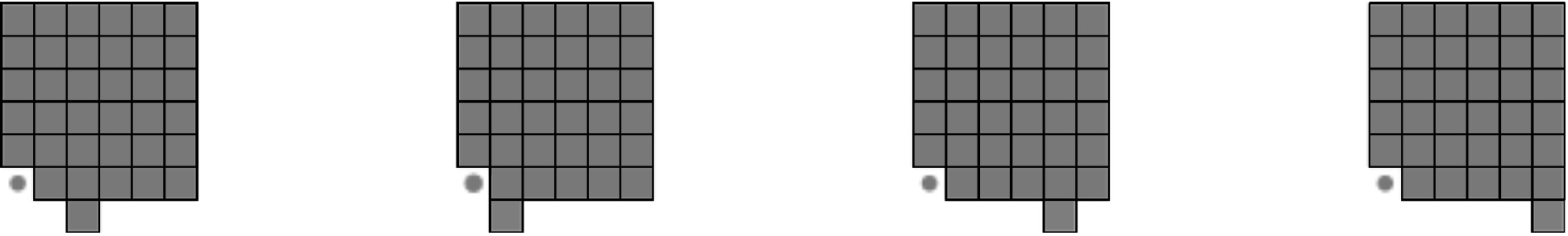}
\end{center}
\caption{Four among the five configurations of the set $\mc
  E^{0,0}_{\mb x}$ for $n=6$. The gray dot indicates the site $\mb
  x$. We placed a square $[-1/2,1/2)^2$ around each particle.}
\label{fig2a}
\end{figure} 

The process $\{\eta^\beta_t : t\ge 0\}$ can reach any configuration
$\xi\in \mc E^{i,j}_{\mb x}$ from any configuration $\eta\in \mc
E^{i,j}_{\mb x}$ with rate one jumps, while any jump from a
configuration in $\mc E^{i,j}_{\mb x}$ to a configuration which does
not belong to this set has rate at most $e^{-\beta}$. This means that
at low temperatures the process $\eta_t$ reaches equilibrium in $\mc
E^{i,j}_{\mb x}$ before exiting this set, which is the first condition
for a set to be the well of a valley.

The main result of this subsection states that for any configuration
$\xi\in \mc E^{i,j}_{\mb x}$, the triples $(\mc E^{i,j}_{\mb x}, \mc
E^{i,j}_{\mb x}\cup \Delta_1, \xi)$ are valleys in the terminology of
\cite{bl2}. This means, in particular, that starting from any
configuration in $\mc E^{i,j}_{\mb x}$, the hitting time of the set
$\bb H_{01}\setminus\mc E^{i,j}_{\mb x}$ properly rescaled converges
in distribution, as $\beta\uparrow\infty$, to an exponential random
variable. We compute in Proposition \ref{s09} the time scale which
turns the limit a mean one exponential distribution, as well as the
asymptotic distribution of $\eta(H(\bb H_{01}\setminus\mc E^{i,j}_{\mb
  x}))$.

Denote by $\mc N(\mc E^{i,j}_{\mb x})$, $\mc N$ for neighborhood, the
configurations which do not belong to $\mc E^{i,j}_{\mb x}$, but which
can be reached from a configuration in $\mc E^{i,j}_{\mb x}$ by
performing a jump which has rate $e^{-\beta}$. The set $\mc N(\mc
E^{2,2}_{\mb w})$, for instance, has the following $3n$
elements. There are $n+1$ configurations obtained when the top
particle detaches itself from the others: $\sigma^{\mb w_{2}, \mb z}
\eta^{\mb w}$, where $\mb z=(-1,n)$, $(a,n+1)$, $0\le a\le n-2$,
$(n-1,n)$. There are $n-1$ configurations obtained when the particle
at $\mb w_2-e_2$ moves upward: $\sigma^{\mb w_2 - e_2, \mb z}
\eta^{\mb w}$, $\mb z = (a,n)$, $0\le a\le n-2$. There are $n-2$
configurations obtained when the particle at $\mb w_2-e_1$ moves to
the right: $\sigma^{\mb w_2 - e_1, \mb z} \eta^{\mb w}$, $\mb z =
(a,n)$, $0\le a\le n-3$. To complete the description of the set $\mc
N(\mc F^{2,2}_{\mb w})$, we have to add the configurations
$\sigma^{\mb w_3, \mb w_3+e_2}\sigma^{\mb w_2, \mb w_3 +e_1+e_2}
\eta^{\mb w}$ and $\sigma^{\mb w_2 - e_1, \mb w_2 -e_1 +e_2}
\sigma^{\mb w_2, \mb w_2-2e_1+e_2} \eta^{\mb w}$.

\begin{lemma}
\label{s10}
For $\mb x\in\Lambda_L$, $0\le i,j\le 3$, and $\xi\in \mc N(\mc
E^{i,j}_{\mb x})$, there exists a probability measure $\bb M(\xi,
\,\cdot\,)$ defined on $\bb H_{01}$ such that
\begin{equation*}
\lim_{\beta\to\infty} \mb P^\beta_{\xi} 
\big[ \eta(H_{\bb H_{01}}) \in \mc A \big] \;=\; \bb M (\xi,
\mc A)
\end{equation*}
for all $\mc A\subset\bb H_{01}$.  Moreover, let $\Pi$ be one of the
sets $\mc E^{i,j}_{\mb w}$, $0\le i,j\le 3$, or one of the singletons
$\{\eta^{\mb w}\}$, $\{\sigma^{\mb w_2, \mb w_3 +e_1+e_2} \sigma^{\mb
  w_0, \mb w_3 +e_2} \eta^{\mb w}\}$. Then,
\begin{enumerate}
\item For $\mb z\in J_1= \{(-1,n), (n-1,n), (a,n+1) : 0\le a\le
  n-2\}$,
\begin{equation*}
\begin{split}
\bb M (\sigma^{\mb w_{2}, \mb z} \eta^{\mb w}, \Pi)\;&=\; 
\sum_{k=0}^3 \mf p(\mb z, Q^{2,k}_{\mb w}) \, 
\mb 1\{\Pi = \mc E^{2,k}_{\mb w}\} \\
& +\; \mf p(\mb z, \mb w_2 + e_2)\, \bb M_1(\Pi) \, 
\;+\; \mf p(\mb z, \mb w_2 + e_1)\, \bb M_2(\Pi)\;.
\end{split}
\end{equation*}
\item For $\mb z \in Q^{2,2}_{\mb w}$,
\begin{equation*}
\begin{split}
\bb M (\sigma^{\mb w_2 - e_2, \mb z} \eta^{\mb w}, \Pi)\; &=\;
\frac{1}{n-1} \, \mb 1\{\Pi = \mc E^{1,2}_{\mb w} \} \;+\;
\mf r^{0}_{n}(\mf n_{\mb z}) \, \bb M_1(\Pi) \\
\; &+\; \Big \{\frac{n-2}{n-1} - \mf r^{0}_{n}(\mf n_{\mb z}) 
\Big\}\, \mb 1\{\Pi = \mc E^{2,2}_{\mb w} \} \;.
\end{split}
\end{equation*}
where $\mf n_{\mb z}= n-1- z_1$, $\mb z = (z_1, z_2)$. 
\item For $\mb z = (k,n)$, $0\le k\le n-3$,
\begin{equation*}
\bb M (\sigma^{\mb w_2 - e_1, \mb z} \eta^{\mb w}, \mc E^{2,2}_{\mb w})
\;=\; 1 
\end{equation*}
\item Finally, for the last two configurations of $\mc N(\mc
  E^{2,2}_{\mb w})$, 
\begin{equation*}
\begin{split}
\bb M (\sigma^{\mb w_3, \mb w_3+e_2}\sigma^{\mb w_2, \mb w_3
  +e_1+e_2} \eta^{\mb w}, \Pi)\; &=\; 
\frac 1{n} \, \mb 1\big\{ \Pi = \{\sigma^{\mb w_2, \mb w_3 +e_1+e_2}
\sigma^{\mb w_0, \mb w_3 +e_2}\eta^{\mb w}\} \big\} \\
& +\; \frac{n-1}{n} \, \mb 1\{ \Pi = \mc E^{2,2}_{\mb w}\}\;,
\end{split}
\end{equation*}
\begin{equation*}
\bb M (\sigma^{\mb w_2 - e_1, \mb w_2 -e_1
  +e_2} \sigma^{\mb w_2, \mb w_2-2e_1+e_2} \eta^{\mb w}, \mc
E^{2,2}_{\mb w})\;=\; 1\;.
\end{equation*}
\end{enumerate}
\end{lemma}

\begin{proof}
We present the proof for $i=j=2$, $\mb x=\mb w$, the other cases being
analogous.  As we have seen, the set $\mc N(\mc E^{2,2}_{\mb w})$ has five
different types of configurations. We examine each one separately.
Assume first that $\xi = \sigma^{\mb w_{2}, \mb z} \eta^{\mb w}$ for
some $\mb z\in J_1$, $J_1$ the set defined in the statement of the
lemma.  The free particle, initially at $\mb z$ performs a rate one,
nearest-neighbor, symmetric random walk in $\Lambda_L$ until it
reaches the outer boundary of the square $Q$. All the other possible
jumps have rate at most $e^{-\beta}$ and may therefore be neglected in
the argument.  When the free particle attains $\partial_+ Q$, the
configuration either belongs to one of the sets $\mc E^{2,k}$, $0\le
k\le 3$, or is one of the configurations $\eta^\star_1$ or $\eta^\star_2$
introduced in Lemma \ref{bs01}. By definition \eqref{b09} of $\mf p$,
it belongs to $\mc E^{2,k}_{\mb w}$ with probability $\mf p(\mb z,
Q^{2,k}_{\mb w})$ and is equal to the configuration $\eta^\star_1$
(resp. $\eta^\star_2$) with probability $\mf p(\mb z, \mb w_2 + e_2)$
(resp. $\mf p(\mb z, \mb w_2 + e_1)$). This proves the first assertion
of the lemma.\smallskip

Assume now that $\xi = \sigma^{\mb w_2 - e_2, \mb z} \eta^{\mb w}$,
$\mb z\in Q^{2,2}_{\mb w}$.  The configuration $\xi$ has a particle
attached to the top side of the square $Q$ and a hole on the right
side of the square. This pair behaves as the process $\mb z_t$
introduced in the begining of this section and evolves until the hole
reaches the bottom of the square or its original position at the top.
There are three cases to be considered. The hole may reach $\mb w_1$
before $\mb w_2$. This happens with probability $(n-1)^{-1}$ and the
configuration attained belongs to the equivalent class $\mc
E^{1,2}_{\mb w}$.

The hole may reach $\mb w_2$ before $\mb w_1$ when the top particle is
not at $\mb w_2+e_2$. In this case the process reached a configuration
in $\mc E^{2,2}_{\mb w}$. This event has probability $\bb P^{\mb
  z}_{(j,1)} [ H_{\partial E_n} = H_{E^-_N}]$, where $j=n-1-\mb z
\cdot e_1$ and where $\bb P^{\mb z}_{(j,1)}$ is the measure introduced
in \eqref{b03}. By definition of the set $\partial E_n$,
\begin{equation*}
\begin{split}
\bb P^{\mb z}_{(j,1)} \big[ H_{\partial E_n} = H_{E^-_N} \big] \; &=\;
1 \;-\; \bb P^{\mb z}_{(j,1)} \big[ H_{\partial E_n} = H_{E^+_N} \big]
\;-\; \bb P^{\mb z}_{(j,1)} \big[ H_{\partial E_n} = H_{(0,0)} \big] \\
\;& =\; \frac{n-2}{n-1} \;-\; \mf r^{0}_{n}(j)\;.
\end{split}
\end{equation*}

Finally, the hole may reach $\mb w_2$ before $\mb w_1$ at a time where
the top particle is at $\mb w_2+e_2$. In this case the process reached
the configuration $\eta^\star_1$ introduced in the previous lemma. This
event happens with probability $\mf r^{0}_{n}(j)$, which concludes the
proof of the second assertion of the lemma.
\smallskip

In the case $\xi = \sigma^{\mb w_2 - e_1, \mb z} \eta^{\mb w}$, the
hole initially at $\mb w_2-e_1$ performs a horizontal, rate one,
symmetric random walk on the interval $\{m+1, \dots, n-1\}$, where $m$
represents the horizontal position of the top particle, which itself
performs a horizontal, rate one, symmetric random walk limited on its
right by the hole in the row below. This coupled system evolves as the
process $\mb y_t$ introduced in \eqref{b02} until the hole initially
at $\mb w_2-e_1$ reaches its original position at $\mb
w_2$. \smallskip

Suppose that $\xi = \sigma^{\mb w_3, \mb w_3+e_2}\sigma^{\mb w_2, \mb
  w_3 +e_1+e_2} \eta^{\mb w}$. In this situation the hole at $\mb w_3$
performs a vertical, rate one, symmetric random walk on $\{(0,b) :
0\le b\le n\}$. The hole reaches $\mb w$ before it reaches $\mb w_3$
with probability $n^{-1}$.  \smallskip

Finally, if $\xi = \sigma^{\mb w_2 - e_1, \mb w_2 -e_1 +e_2}
\sigma^{\mb w_2, \mb w_2-2e_1+e_2} \eta^{\mb w}$, there is only one
rate one jump which drives the system back to the set $\mc
E^{2,2}_{\mb w}$.
\end{proof}

By symmetry, the distribution of $\eta(H (\bb H_{01}\setminus\mc
E^{i,\cdot}_{\mb x}))$ can be obtained from the one of $\eta(H (\bb
H_{01}\setminus\mc E^{0,j}_{\mb w}))$, $0\le j\le 3$.  When the set
$\Pi$ is a singleton $\{\zeta\}$ we represent $\bb M (\xi, \{\zeta\})$
by $\bb M(\xi, \zeta)$. This convention is adopted for all functions
of sets without further comment.  Recall the notation introduced in
the beginning of this section and in the statement of Lemma
\ref{s10}. Let
\begin{equation}
\label{b11}
Z(\mc E^{i,j}_{\mb x}) \;=\; \sum_{\xi\in \mc N(\mc E^{i,j}_{\mb x})}
\bb M (\xi, (\mc E^{i,j}_{\mb x})^c) 
\;=\; \sum_{\xi\in \mc N(\mc E^{i,j}_{\mb x})} 
\{ 1 - \bb M (\xi, \mc E^{i,j}_{\mb x})\} \;.
\end{equation}
Note that $Z(\mc E^{i,j}_{\mb x})$ does not depend on $\mb x$. In this
sum, the terms $\bb M (\xi, \,\cdot\,)$ are not multiplied by weights
$\omega (\xi)$ because asymptotically the process hits $\mc N(\mc
E^{i,j}_{\mb x})$ according to a uniform distribution. In view of the
previous lemma and by \eqref{b10},
\begin{equation*}
\begin{split}
& Z(\mc E^{2,2}_{\mb x}) \;=\; 1\;+\; \frac 1{n-1} \;+\;
(1+\mf r^-_n) \, [1-\bb M_1(\mc E^{2,2}_{\mb w})] \\
&\quad \;+\; 
\sum_{\mb z\in J^\star_1} \Big\{ [1-\mf p(\mb z, Q^{2,2}_{\mb w})]
\,-\, \mf p(\mb z, \mb w_2 + e_2) \bb M_1(\mc E^{2,2}_{\mb w}) \,-\, 
\mf p(\mb z, \mb w_2 + e_1) \bb M_2(\mc E^{2,2}_{\mb w})\Big\}\;,     
\end{split}
\end{equation*} 
where $J^\star_1 = J_1 \setminus\{\mb w_2+e_2\}$.

\begin{proposition}
\label{s09}
Fix $0\le i,j\le 3$ and $\mb x\in\Lambda_L$.
\begin{enumerate}
\item For any $\xi\in \mc E^{i,j}_{\mb x}$, the triple $(\mc
  E^{i,j}_{\mb x}, \mc E^{i,j}_{\mb x}\cup \Delta_1, \xi)$ is a valley
  of depth $\mu_K (\mc E^{i,j}_{\mb x})/$ $\Cap_K(\mc
  E^{i,j}_{\mb x}, [\mc E^{i,j}_{\mb x}\cup \Delta_1]^c)$;

\item For any $\xi\in \mc E^{i,j}_{\mb x}$, under $\mb P^\beta_{\xi}$,
  $H (\bb H_{01}\setminus\mc E^{i,j}_{\mb x})/e^\beta$ converges in
  distribution to an exponential random variable of parameter $Z(\mc
  E^{i,j}_{\mb x})/ |\mc E^{i,j}_{\mb x}|$;

\item For any $\xi\in \mc E^{i,j}_{\mb x}$, $\Pi \subset
  \bb H_{01}\setminus\mc E^{i,j}_{\mb x}$, 
\begin{equation*}
\lim_{\beta\to\infty} \mb P^\beta_{\xi}
\big[ \eta (H (\bb H_{01}\setminus\mc E^{i,j}_{\mb x})) \in \Pi \big]
\;=\; \frac 1{Z(\mc E^{i,j}_{\mb x})} \sum_{\eta\in \mc N(\mc E^{i,j}_{\mb x})} 
\bb M (\eta, \Pi) \;=:\; Q(\mc E^{i,j}_{\mb x}, \Pi)\;. 
\end{equation*}
\end{enumerate}
\end{proposition}

\begin{proof}
Recall \cite[Theorem 2.6]{bl2}. Condition (2.15) is fulfilled by
definition of the set $\Delta_1$. A simple argument shows that
$G_K(\xi, \zeta) = e^{-\beta} \mu_K(\eta^{\mb w})$ for any
pair of configurations $\xi\not = \zeta\in \mc E^{i,j}_{\mb x}$, and
that $G_K(\mc E^{i,j}_{\mb x}, [\mc E^{i,j}_{\mb x}\cup
\Delta_1]^c) \le e^{-2\beta} \mu_K(\eta^{\mb w})$. Condition
(2.14) follows from these estimates and \eqref{17}. This
proves the first assertion of the lemma.

To prove the second assertion of the lemma, we start with a recursive
formula for $H_{\bb H_{01} \setminus \mc E^{i,j}_{\mb x}}$. Let $\tau_1$
the time the process leaves the set $\mc E^{i,j}_{\mb x}$: $\tau_1 =
\inf\{t>0 : \eta^\beta_t \not \in \mc E^{i,j}_{\mb x}\}$. We have that
\begin{equation*}
H_{\bb H_{01} \setminus \mc E^{i,j}_{\mb x}} \; =\; \tau_1 \;+\; 
H_{\bb H_{01}} \circ \theta_{\tau_1} \;+\;
\mb 1\{ H_{\bb H_{01}} \circ \theta_{\tau_1} = H_{\mc E^{i,j}_{\mb x}} 
\circ \theta_{\tau_1} \} \, 
H_{\bb H_{01} \setminus \mc E^{i,j}_{\mb x}} \circ \theta_{H^+_{\bb H_{01}}} \;,
\end{equation*}
where $\{\theta_t : t\ge 0\}$ stands for the shift operators.

Fix $\lambda>0$ and let $\lambda_\beta = \lambda e^{-\beta}$. By the
strong Markov property, for any $\xi\in \mc E^{i,j}_{\mb x}$,
\begin{eqnarray}
\label{08}
&& \mb E^\beta_{\xi} \Big[ e^{-\lambda_\beta H_{\bb H_{01}
    \setminus \mc E^{i,j}_{\mb x}}}\Big] \; =\;
\mb E^\beta_{\xi} \Big[ e^{-\lambda_\beta \,\tau_1}\,
\mb E^\beta_{\eta_{\tau_1}} 
\Big[ \mb 1\{ H_{\bb H_{01}} \not = H_{\mc E^{i,j}_{\mb x}} \} \,
e^{-\lambda_\beta H_{\bb H_{01}}} \Big]\, \Big] \\
&& \quad +\;
\mb E^\beta_{\xi} \Big[ e^{-\lambda_\beta \,\tau_1}\,
\mb E^\beta_{\eta_{\tau_1}} 
\Big[ \mb 1\{ H_{\bb H_{01}} = H_{\mc E^{i,j}_{\mb x}} \} \,
e^{-\lambda_\beta H_{\bb H_{01}}}\, \exp\big \{-\lambda_\beta H_{\bb H_{01} 
\setminus \mc E^{i,j}_{\mb x}} \circ \theta_{H_{\bb H_{01}}} \big\} 
\Big]\, \Big]  \;. \nonumber
\end{eqnarray}

Recall the definition of $\mc N(\mc E^{i,j}_{\mb x})$ given just
before the statement of Lemma \ref{s10}.  With a probability which
converges to $1$ as $\beta\uparrow\infty$, $\eta^\beta_{\tau_1}$
belongs to $\mc N(\mc E^{i,j}_{\mb x})$. Each configuration in $\mc
N(\mc E^{i,j}_{\mb x})$ belongs to an equivalent class which
eventually attains $\bb H_{01}$ after a finite random number of rate
one jumps.  This proves that
\begin{equation*}
\lim_{A\to\infty} \lim_{\beta\to\infty} \max_{\zeta \in \mc N(\mc E^{i,j}_{\mb x})} 
\mb P^\beta_{\zeta} \big[ H_{\bb H_{01}} > A \big] \;=\; 0\;.
\end{equation*}
Therefore, we may replace in \eqref{08} $\exp\{-\lambda_\beta
H_{\bb H_{01}}\}$ by $1$ at a cost which vanishes as
$\beta\uparrow\infty$.

By the strong Markov property, after the last replacement, the second
term on the right hand side of \eqref{08} can be rewritten as
\begin{equation*}
\mb E^\beta_{\xi} \Big[ e^{-\lambda_\beta \,\tau_1}\,
\mb E^\beta_{\eta_{\tau_1}} 
\Big[ \mb 1\{ H_{\bb H_{01}} = H_{\mc E^{i,j}_{\mb x}} \} \,
\mb E^\beta_{\eta_{H_{\bb H_{01}}}} \big[ 
\exp\big \{-\lambda_\beta H_{\bb H_{01} 
\setminus \mc E^{i,j}_{\mb x}}  \big\} \big]\,
\Big]\, \Big] \;.
\end{equation*}
Since $\mc E^{i,j}_{\mb x}$ is an equivalent class and the process
leaves $\mc E^{i,j}_{\mb x}$ only after a rate $e^{-\beta}$ jump, a
simple coupling argument shows that
\begin{equation*}
\lim_{\beta\to\infty} \max_{\eta ,\zeta \in \mc E^{i,j}_{\mb x}}
\Big| \, \mb E^\beta_{\eta} \big[ \exp\big \{-\lambda_\beta H_{\bb H_{01} 
\setminus \mc E^{i,j}_{\mb x}}  \big\} \big] \;-\;
\mb E^\beta_{\zeta} \big[ \exp\big \{-\lambda_\beta H_{\bb H_{01} 
\setminus \mc E^{i,j}_{\mb x}}  \big\} \big]  \, \Big|\; =\; 0\;. 
\end{equation*}
The previous expectation is thus equal to
\begin{equation*}
\mb E^\beta_{\xi} \big[ 
\exp\big \{-\lambda_\beta H_{\bb H_{01} 
\setminus \mc E^{i,j}_{\mb x}}  \big\} \big]\,
\mb E^\beta_{\xi} \Big[ e^{-\lambda_\beta \,\tau_1}\,
\mb E^\beta_{\eta_{\tau_1}} 
\big[ \mb 1\{ H_{\bb H_{01}} = H_{\mc E^{i,j}_{\mb x}} \} \,
\big]\, \Big] 
\end{equation*}
plus an error which vanishes as $\beta\uparrow\infty$.

We claim that $(e^{-\beta} \tau_1, \eta^\beta_{\tau_1})$ converges in
distribution, as $\beta\uparrow\infty$, to a pair of independent
random variables where the first coordinate is an exponential time and
the second coordinate has a distribution concentrated on $\mc N(\mc
E^{i,j}_{\mb x})$. The proof of this claim relies on \cite[Theorem
2.7]{bl2} and on a coupling argument.

Let $G^{i,j}_{\mb x} = \mc N(\mc E^{i,j}_{\mb x}) \cup \mc
E^{i,j}_{\mb x}$. Consider the Markov process $\{\hat \eta^\beta_t:
t\ge 0\}$ on $G^{i,j}_{\mb x}$ whose jump rates $\hat r(\eta,\xi)$ are
given by
\begin{equation*}
\hat r(\eta,\xi) \;=\; 
\begin{cases}
  r(\eta,\xi) & \text{if $\eta\in \mc E^{i,j}_{\mb x}$, 
$\xi\in G^{i,j}_{\mb x}$}\;, \\
  r(\xi,\eta) & \text{if $\eta\in \mc N(\mc E^{i,j}_{\mb x})$, $\xi\in \mc
    E^{i,j}_{\mb x}$} \;, \\
0 & \text{otherwise} \;.
\end{cases}
\end{equation*}
Note that $\hat r(\eta,\xi) = e^{-\beta}$ or $0$ if $\eta\in \mc N(\mc
E^{i,j}_{\mb x})$, and that we may couple the processes $\eta^\beta_t$
and $\hat \eta^\beta_t$ in such a way that the probability of the
event $\{\eta^\beta_t = \hat \eta^\beta_t : 0\le t\le \tau_1\}$
converges to one as $\beta\uparrow\infty$ if the initial state belongs
to $\mc E^{i,j}_{\mb x}$.

Let $\{\xi^1, \dots, \xi^m\}$ be an enumeration of the set $\mc N(\mc
E^{i,j}_{\mb x})$ and consider the partition $\mc E^{i,j}_{\mb x} \cup
\{\xi^1\} \cup \dots \cup\{\xi^m\}$ of the set $G^{i,j}_{\mb
  x}$. Assumption (H1) of \cite[Theorem 2.7]{bl2} for the process
$\hat \eta^\beta_t$ is empty for the sets $\{\xi^j\}$ and has been
checked in the first part of this proof for the set $\mc E^{i,j}_{\mb
  x}$. Assumption (H0) for the process $\hat \eta^\beta_t$ speeded up
by $e^\beta$ can be verified by a direct computation. Therefore, by
\cite[Theorem 2.7]{bl2}, the pair $(e^{-\beta} \hat\tau_1, \hat
\eta^\beta_{\tau_1})$ converges in distribution, as
$\beta\uparrow\infty$, to a pair of independent random variables in
which the first coordinate has an exponential distribution and the
second one is concentrated over $\mc N(\mc E^{i,j}_{\mb x})$.  This
result can be extended to the original pair $(e^{-\beta} \tau_1,
\eta^\beta_{\tau_1})$ by the coupling argument alluded to above.

It follows from the claim just proved and the previous estimates that
\begin{equation*}
\lim_{\beta\to\infty} 
\mb E^\beta_{\xi} \Big[ e^{-\lambda_\beta H_{\bb H_{01}
    \setminus \mc E^{i,j}_{\mb x}}}\Big] \; =\;
\lim_{\beta\to\infty} 
\frac{\mb E^\beta_{\xi} 
\big[ e^{-\lambda_\beta \,\tau_1}\,\big] \,
\mb E^\beta_{\xi} \Big[
\mb P^\beta_{\eta_{\tau_1}} 
\big[ H_{\bb H_{01}} \not = H_{\mc E^{i,j}_{\mb x}} \big]\, \Big]}
{1- \mb E^\beta_{\xi} 
\big[ e^{-\lambda_\beta \,\tau_1}\,\big] \,
\mb E^\beta_{\xi} \Big[ \mb P^\beta_{\eta_{\tau_1}} 
\big[ H_{\bb H_{01}} = H_{\mc E^{i,j}_{\mb x}} \big]\, \Big]}\;\cdot 
\end{equation*}
If $\tau_1/e^\beta$ converges to an exponential random variable of
parameter $\theta$, the right hand side becomes
\begin{equation*}
\lim_{\beta\to\infty} 
\frac{ \theta \, \mb E^\beta_{\xi} \Big[ \mb P^\beta_{\eta_{\tau_1}} 
\big[ H_{\bb H_{01}} \not = H_{\mc E^{i,j}_{\mb x}} \big]\, \Big]}
{\lambda + \theta\, 
\mb E^\beta_{\xi} \Big[ \mb P^\beta_{\eta_{\tau_1}} 
\big[ H_{\bb H_{01}} \not = H_{\mc E^{i,j}_{\mb x}} \big]\, \Big]}\;,
\end{equation*}
which means that $H(\bb H_{01} \setminus \mc E^{i,j}_{\mb x})/e^\beta$
converges to an exponential random variable of parameter 
\begin{equation*}
\gamma \;=\; \theta \, \lim_{\beta\to\infty} 
\mb E^\beta_{\xi} \Big[ \mb P^\beta_{\eta_{\tau_1}} 
\big[ H_{\bb H_{01}} \not = H_{\mc E^{i,j}_{\mb x}} \big]\, \Big]\;.
\end{equation*}

We examine the case $i=j=2$, $\mb x =\mb w$.  Recall the description
of the set $\mc N(\mc E^{2,2}_{\mb w})$ presented before Lemma
\ref{s10}.  By computing the average rates which appear in assumption
(H0) of \cite{bl2}, we obtain that under $\mb P^\beta_{\xi}$,
$\tau_1/e^\beta$ converges in distribution to an exponential random
variable of parameter $|\mc N(\mc E^{2,2}_{\mb w})|/|\mc E^{2,2}_{\mb
  w}| = 3n/(n-1)$, and that $\eta^\beta_{\tau_1}$ converges to a
uniform distribution on $\mc N(\mc E^{2,2}_{\mb w})$. Hence, by the
conclusions of the previous paragraph and by Lemma \ref{s10}, $H(\bb
H_{01} \setminus \mc E^{2,2}_{\mb w})/e^\beta$ converges to an
exponential random variable of parameter $Z(\mc E^{2,2}_{\mb w})/ |\mc
E^{2,2}_{\mb w}|$.  This proves the second assertion of the
proposition.

We turn to the third assertion. Denote by $\{H_j : j\ge 1\}$ the
successive return times to $\bb H_{01}$:
\begin{equation*}
H_1 = H^+(\bb H_{01})\;, \quad H_{j+1} = H^+(\bb H_{01}) \circ \theta_{H_j}\;,
\quad j\ge 1\;.
\end{equation*}
With this notation, we may write for every $\xi\in \mc E^{2,2}_{\mb w}$,
\begin{equation}
\label{09}
\mb P^\beta_{\xi} \big[ \eta(H_{\bb H_{01}\setminus \mc E^{2,2}_{\mb w}}) 
\in \Pi \big] \; =\; \sum_{j\ge 1} \mb P^\beta_{\xi} 
\big[ \eta(H_k) \in \mc E^{2,2}_{\mb w} \;, 1\le k\le j-1\;,
\eta(H_j) \in \Pi \big]\;.
\end{equation}

By the strong Markov property, for any $\xi'\in \mc E^{2,2}_{\mb w}$,
$\Pi'\subset\bb H_{01}$,
\begin{equation*}
\mb P^\beta_{\xi'} \big[ \eta(H_1) \in \Pi' \big] \; =\;
\mb E^\beta_{\xi'} \Big[ \mb P^\beta_{\eta_{\tau_1}} 
\big[\eta(H_{\bb H_{01}}) \in \Pi'\big]\, \Big]\;.
\end{equation*}
Under $\mb P^\beta_{\xi}$, the distribution of $\eta_{\tau_1}$
converges to the uniform distribution over $\mc N(\mc E^{2,2}_{\mb
  w})$ as $\beta\uparrow\infty$. Hence, by Lemma \ref{s10},
\begin{equation*}
\lim_{\beta\to\infty} \mb P^\beta_{\xi'} \big[ 
\eta(H_{1}) \in \Pi' \big] \;=\; \frac 1{|\mc N(\mc E^{2,2}_{\mb w})|} 
\sum_{\eta\in \mc N(\mc E^{2,2}_{\mb w})} \bb M (\eta, \Pi')
\end{equation*}
for all $\xi'\in \mc E^{2,2}_{\mb w}$, $\Pi'\subset\bb H_{01}$. Denote the
right hand side of the previous formula by $q(\Pi')$. It follows from
identity \eqref{09}, the strong Markov property and the previous
observation that for all $\Pi\subset \bb H_{01}\setminus \mc E^{2,2}_{\mb
  w}$,
\begin{equation*}
\lim_{\beta\to\infty} \mb P^\beta_{\xi} \big[ \eta(H_{\bb H_{01}\setminus
  \mc E^{2,2}_{\mb w}}) \in \Pi \big] \;=\; \frac{q(\Pi)}
{1-q(\mc E^{2,2}_{\mb w})}\;,
\end{equation*}
which concludes the proof of the proposition.
\end{proof}

Since $\bb H_{01}\setminus \mc E^{i,j}_{\mb x} = [\mc E^{i,j}_{\mb x}
\cup \Delta_1]^c$, it follows from the second assertion of the
proposition that the depth of the valley $(\mc E^{i,j}_{\mb x}, \mc
E^{i,j}_{\mb x}\cup \Delta_1, \xi)$ is $e^\beta|\mc E^{i,j}_{\mb
  x}|/Z(\mc E^{i,j}_{\mb x})$. In particular,
\begin{equation}
\label{19}
\lim_{\beta\to\infty} \frac{\mu_K (\mc E^{i,j}_{\mb
   x})}{e^\beta\, \Cap_K(\mc E^{i,j}_{\mb x}, [\mc E^{i,j}_{\mb x}\cup
 \Delta_1]^c)} \;=\; \frac {|\mc E^{i,j}_{\mb x}|}
{Z(\mc E^{i,j}_{\mb x})}\;\cdot
\end{equation}
Since $\mu_K (\mc E^{i,j}_{\mb x}) = |\mc E^{i,j}_{\mb x}|\,
e^{-\beta}\, \mu_K(\eta^{\mb w})$, 
\begin{equation*}
\lim_{\beta\to\infty} \frac
{\Cap_K(\mc E^{i,j}_{\mb x}, [\mc E^{i,j}_{\mb x}\cup
 \Delta_1]^c)} {e^{-2\beta}\, \mu_K(\eta^{\mb w})}
\;=\; Z(\mc E^{i,j}_{\mb x})\;.
\end{equation*}

In view of Lemma \ref{s10}, we have the following explicit formula for
the probability measure $Q(\mc E^{2,2}_{\mb w}, \,\cdot\,)$ on $\bb
H_{01}$. Let 
\begin{equation}
\label{b12}
\mb R(\mc E^{i,j}_{\mb x}, \Pi) \;=\;
Z(\mc E^{i,j}_{\mb x})\, Q(\mc E^{i,j}_{\mb x}, \Pi) \;=\; 
\sum_{\eta\in \mc N(\mc E^{i,j}_{\mb x})} \bb M (\eta, \Pi) \;, \quad
\Pi\subset \bb H_{01}\;. 
\end{equation}
Recall the definition of the set $J^\star_1$ introduced one equation below
\eqref{b11}. Then,
\begin{equation*}
\begin{split}
\mb R (\mc E^{2,2}_{\mb w}, \eta^{\mb w}) 
\;& =\; (1+\mf r^-_n) \, \bb M_1(\eta^{\mb w}) \\
& +\; \sum_{\mb z\in J^\star_1} \Big\{ \mf p(\mb z, \mb w_2 + e_2) 
\, \bb M_1(\eta^{\mb w}) \,+\, \mf p(\mb z, \mb w_2 + e_1) 
\bb M_2(\eta^{\mb w})\Big\}\;,
\end{split}
\end{equation*}
for $0\le i,j\le 3$, $(i,j)\not = (2,2)$;
\begin{equation*}
\begin{split}
\mb R(\mc E^{2,2}_{\mb w}, \mc E^{i,j}_{\mb w}) \;& =\;
(1+\mf r^-_n) \, \bb M_1(\mc E^{i,j}_{\mb w})
\;+\; \mb 1\{\mc E^{i,j}_{\mb w}=\mc E^{1,2}_{\mb w}\} 
\;+\; \sum_{\mb z\in J^\star_1} \mf p(\mb z, \mb w_2 + e_2) 
\, \bb M_1(\mc E^{i,j}_{\mb w})\\
& +\; \sum_{\mb z\in J^\star_1} \Big\{ \mf p(\mb z, Q^{2,j}_{\mb w}) \, \mb
1\{\Pi=\mc E^{2,j}_{\mb w}\} \,+\, 
\mf p(\mb z, \mb w_2 + e_1) \bb M_2(\mc E^{i,j}_{\mb w})\Big\}\;;
\end{split}
\end{equation*}
and 
\begin{equation*}
\mb R (\mc E^{2,2}_{\mb w}, 
\sigma^{\mb w_2, \mb w_3 +e_1+e_2}\sigma^{\mb w_0, \mb w_3 +e_2} \eta^{\mb w}) 
\;=\; \frac 1n\;\cdot 
\end{equation*}
The rate $\mb R (\mc E^{2,2}_{\mb w}, \Pi)$ vanishes if $\Pi$ does not
intersect $\{\eta^{\mb w}, \xi_1\}\cup_{(i,j)\not = (2,2)} \mc
E^{i,j}_{\mb w}$, where $\xi_1$ is the configuration appearing in the
previous displayed formula.  Hence, on the time scale $e^\beta$,
starting from the valley $\mc E^{2,2}_{\mb w}$ the process may fall in
the deep well $\eta^{\mb w}$, it may reach some valley $\mc
E^{i,j}_{\mb w}$, $(i,j)\not = (2,2)$, which are similar to $\mc
E^{2,2}_{\mb w}$, or attain the configuration $\sigma^{\mb w_2, \mb
  w_3 +e_1+e_2} \sigma^{\mb w_0, \mb w_3 +e_2}\eta^{\mb w}$. In the
next subsection we show that this configuration is the well of a
valley, a property shared by a class of configurations.

\subsection{The valleys $\{\eta_{\mb x}^{\mf a, (\mb
  k,\bs \ell)}\}$}
\label{ss2}

Let $R^{\mf l}$, $R^{\mf s}$ be the rectangles $R^{\mf l} =\{1,\dots,
n-1\}\times\{1, \dots, n-2\}$, $R^{\mf s} = \{1, \dots, n-2\} \times
\{1,\dots, n-1\}$, where $\mf l$ stands for lying and $\mf s$ for
standing. Let $n_0^{\mf s} = n_2^{\mf s} = n-2$, $n_1^{\mf s} =
n_3^{\mf s} = n-1$ be the length of the sides of the standing
rectangle $R^{\mf s}$. Similarly, denote by $n_i^{\mf l}$, $0\le i\le
3$, the length of the sides of the lying rectangle $R^{\mf l}$:
$n_i^{\mf l} = n_{i+1}^{\mf s}$, where the sum over the index $i$ is
performed modulo $4$.

Denote by $\bb I_{\mf a}$, $\mf a \in \{\mf s ,\mf l\}$, the set of
pairs $(\mb k, \bs \ell) = (k_0,\ell_0; k_1,\ell_1; k_2,\ell_2;
k_3,\ell_3)$ such that
\begin{itemize}
\item $0\le k_i \le \ell_i \le n_i^{\mf a}$,
\item If $k_j=0$, then $\ell_{j-1}= n_{j-1}^{\mf a}$.
\end{itemize}
For $(\mb k,\bs \ell) \in \bb I_{\mf a}$, $\mf a \in\{ \mf s, \mf
l\}$, let $R^{\mf l}(\mb k,\bs \ell)$, $R^{\mf s}(\mb k,\bs \ell)$ be
the sets
\begin{equation*}
\begin{split}
R^{\mf l} (\mb k,\bs \ell)\; & =\;
R^{\mf l} \; \cup\; \{(a,0) : k_0\le a\le \ell_0\} \; \cup\;
\{(n,b) : k_1\le b\le \ell_1\}\;\cup \\
& \cup\; \{(n-a,n-1) : k_2\le a\le \ell_2\}
\;\cup\; \{(0,n-1-b) : k_3\le b\le \ell_3\}\;, \\
R^{\mf s} (\mb k,\bs \ell)  \;& =\;
R^{\mf s} \; \cup\; \{(a,0) : k_0\le a\le \ell_0\} \; \cup\;
\{(n-1,b) : k_1\le b\le \ell_1\} \;\cup \\
& \cup\; \{(n-1-a,n) : k_2\le a\le \ell_2\}
\;\cup\; \{(0,n-b) : k_3\le b\le \ell_3\} \;.
\end{split}
\end{equation*}
Note that a hole between particles on the side of a rectangle is not
allowed in the sets $R^{\mf a} (\mb k,\bs \ell)$, $R^{\mf s} (\mb
k,\bs \ell)$. 

Denote by $I_{\mf a}$, $\mf a \in \{ \mf s, \mf l\}$, the set of pairs
$(\mb k, \bs \ell)\in \bb I_{\mf a}$ such that $|R^{\mf a} (\mb k,\bs
\ell)|=n^2$. For $(\mb k,\bs \ell) \in I_{\mf a}$, denote by $M_i(\mb
k,\bs \ell)$ the number of particles attached to the side $i$ of the
rectangle $R^{\mf a} (\mb k,\bs \ell)$:
\begin{equation*}
M_i(\mb k,\bs \ell) \;=\; 
\begin{cases}
\ell_i-k_i+1 & \text{if $k_{i+1}\ge 1\;,$}\\
\ell_i-k_i+2 & \text{if $k_{i+1}=0$}\ .
\end{cases}
\end{equation*}
Clearly, for $(\mb k, \bs \ell)\in I_{\mf a}$, $\sum_{0\le i\le 3}
M_i(\mb k,\bs \ell) = 3n-2 + A$, where $A$ is the number of occupied
corners, which are counted twice since they are attached to two sides.

Denote by $I^*_{\mf a} \subset I_{\mf a}$, the set of pairs $(\mb k,
\bs \ell)\in I_{\mf a}$ whose rectangles $R^{\mf a} (\mb k,\bs \ell)$
have at least two particles on each side: $M_i(\mb k,\bs \ell)\ge 2$,
$0\le i\le 3$. Note that if $(\mb k, \bs \ell)$ belongs to $I^*_{\mf
  a}$, for all $\mb x\in R^{\mf a} (\mb k,\bs \ell)$, there exist $\mb
y$, $\mb z \in R^{\mf a} (\mb k,\bs \ell)$, $\mb y\not = \mb z$, with
the property $\Vert \mb x-\mb y\Vert =\Vert \mb x-\mb z\Vert =1$.

For $(\mb k,\bs \ell) \in I_{\mf a}$, $\mf a \in\{ \mf s, \mf l\}$,
$\mb x\in \Lambda_L$, let $R^{\mf a}_{\mb x} (\mb k,\bs \ell) = \mb x
+ R^{\mf a} (\mb k,\bs \ell)$, and let $\eta_{\mb x}^{\mf a, (\mb
  k,\bs \ell)}$ represent the configurations defined by
\begin{equation*}
\text{$\eta_{\mb x}^{\mf a, (\mb k,\bs \ell)} (a,b) = 1$
    if and only if $(a,b)\in R_{\mb x}^{\mf a}(\mb k,\bs \ell)$}\;.
\end{equation*}
The configurations $\eta_{\mb x}^{\mf a, (\mb k,\bs \ell)}$, $(\mb
k,\bs \ell) \in I_{\mf a} \setminus I^*_{\mf a}$, belong to $\Omega^1$
or form a $(n-1)\times(n+1)$ rectangle of particles with one extra
particle attached to a side of length $n+1$.  Let $\Omega^2 =
\Omega^2_{L,K}$, be the set of configurations associated to the pairs
$(\mb k,\bs \ell)$ in $I^*_{\mf a}$:
\begin{equation*}
\Omega^2_{\mb x} \;=\;  \{ \eta_{\mb x}^{\mf a, (\mb k,\bs \ell)} : 
\mf a \in\{ \mf s, \mf l\}, (\mb k,\bs \ell) \in I^*_{\mf a} \} \;, 
\quad
\Omega^2 \;=\;  \bigcup_{\mb x\in \Lambda_L} \Omega^2_{\mb x} \;.
\end{equation*}

\begin{figure}[!h]
\begin{center}
\includegraphics[scale =0.5]{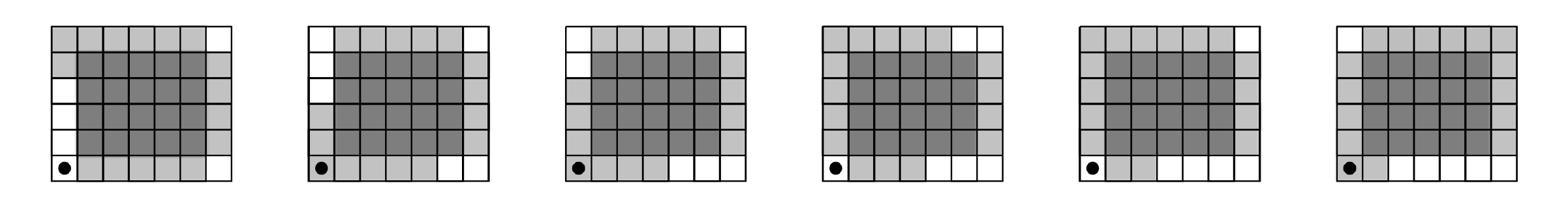}
\end{center}
\caption{Some configurations $\eta_{\mb x}^{\mf l, (\mb k,\bs \ell)}$
  for $n=6$. The first one corresponds to the vector $(\mb k,\bs
  \ell)= ((1,5); (1,4);(1,6); (0,1))$ and the last one to the vector
  $(\mb k,\bs \ell)= ((0,1); (1,5);(0,5);(1,5))$. The inner gray
  rectangle represents the set $\mb x + R^{\mf l}$ and the black dot
  the site $\mb x$.}
\end{figure}

To describe the valleys which can be attained from $\eta_{\mb x}^{\mf
  a, (\mb k,\bs \ell)}$ we have to define a map from $\Omega^2$ to
$\bb H_{01}$ which translates by one unit all particles in an external row
or column of a rectangle $R^{\mf a}_{\mb x} (\mb k,\bs \ell)$. This
must be done carefully because the translation of one row may produce
a configuration which does not belong to $\bb H_{01}$, or a configuration
$\eta_{\mb x}^{\mf a, (\mb k',\bs \ell')}$, where the vector $(\mb
k',\bs \ell')$ differs from $(\mb k,\bs \ell)$ in more than one
coordinate.

Denote by $I_{\mf a,i}^-$ (resp. $I_{\mf a,i}^+$), $0\le i\le 3$, the
pairs $(\mb k, \bs \ell)$ in $I_{\mf a}^*$ for which the particle
sitting at $k_i$ (resp. $\ell_i$) jumps to $k_i-1$ (resp. $\ell_i+1$)
at rate $e^{-\beta}$. The abuse of notation is clear. For instance, by
site $k_0$ we mean the site $(k_0,0)$, or, if $\mf a =\mf s$, by site
$\ell_2$ we mean site $(n-1-\ell_2,n)$. The subsets $I_{\mf a,i}^\pm$
of $I_{\mf a}^*$ are given by
\begin{equation*}
\begin{split}
& I_{\mf a,i}^- \;=\; \{(\mb k, \bs \ell)\in I^*_{\mf a} : 
k_i \ge 2 \text{ or } k_i=1 \,,\, \ell_{i-1} = n_{i-1}^{\mf a}\} \; , 
\\
&\quad
I_{\mf a,i}^+ \;=\; \{(\mb k, \bs \ell)\in I^*_{\mf a} : 
\ell_i \le   n_{i}^{\mf a}-1 \text{ or } \ell_i =  n_{i}^{\mf a}
\,,\, k_{i+1} = 1 \} \;.
\end{split}
\end{equation*}

For $(\mb k, \bs \ell)\in I_{\mf a,i}^-$, denote by $\hat T^-_{\mf
  a,i} \eta_{\mb x}^{\mf a, (\mb k,\bs \ell)}$ the configuration
obtained from $\eta_{\mb x}^{\mf a, (\mb k,\bs \ell)}$ by moving the
particle sitting at $k_i$ to $k_i-1$, with the same abuse of notation
alluded to before. Similarly, for $(\mb k, \bs \ell)\in I_{\mf
  a,i}^+$, denote by $\hat T^+_{\mf a,i} \eta_{\mb x}^{\mf a, (\mb
  k,\bs \ell)}$ the configuration obtained from $\eta_{\mb x}^{\mf a,
  (\mb k,\bs \ell)}$ by moving the particle sitting at $\ell_i$ to
$\ell_i+1$.

Define the map $T^-_{\mf a,i} : I^-_{\mf a,i} \to I_{\mf a}$ by
\begin{equation*}
T^-_{\mf a,i} (\mb k, \bs \ell) \;=
\begin{cases}
(\mb k - \mf e_i , \bs \ell - \mf e_i ) & \text{if $k_{i+1}\ge 1$,} \\
(\mb k - \mf e_i + \mf e_{i+1}, \bs \ell)& \text{if $k_{i+1}=0$,} 
\end{cases}
\end{equation*}
where $\{\mf e_1, \dots, \mf e_4\}$ stands for the canonical basis of
$\bb R^4$. The map $T^+_{\mf a,i} : I^+_{\mf a,i} \to I_{\mf a}$ is
defined in an analogous way. Hence, the map $T^+_{\mf s,2}$ translate
to the \emph{left} all particles on the top row of the rectangle
$R^{\mf s}$ and the map $T^-_{\mf l,3}$ translate in the \emph{upward}
direction all particles on the leftmost column of $R^{\mf l}$.

The vector $T^\pm_{\mf a,i}(\mb k, \bs \ell)$ may not belong to
$I^*_{\mf a}$ when there are only two particles on one side of a
rectangle $R^{\mf a}$ and one of them is translated along another
side. For example, suppose that $k_0=1$, $\ell_0=n-2$, $k_1=0$,
$\ell_1=1$ for a vector $(\mb k, \bs \ell)\in I^*_{\mf s}$. In this
case, necessarily $k_2=1$, $\ell_2 = n-2$, $k_3=0$, $\ell_3=n-1$, and
$T^-_{\mf s,0} (\mb k, \bs \ell)\not \in I^*_{\mf s}$. In fact, the
configuration $\eta^{\mf s, T^-_{\mf s,0} (\mb k, \bs \ell)}_{\mb x}$
belongs to the set $\Omega^3$ to be introduced in the next
subsection. Similarly, if $k_1=2$, $\ell_1=n-1$, $k_2=0$, $\ell_2=1$
for a vector $(\mb k, \bs \ell)\in I^*_{\mf s}$, $T^-_{\mf s,1} (\mb
k, \bs \ell)\not\in I^*_{\mf s}$, and $\eta^{\mf s, T^-_{\mf s,1} (\mb
  k, \bs \ell)}_{\mb x} \in \Omega^1$.

Fix a vector $(\mb k, \bs \ell) \in I^*_{\mf a}$ such that $M_i (\mb
k, \bs \ell) =2$ for some $0\le i\le 3$. Denote by $J_{\mf a, i} (\mb
k, \bs \ell)$ the interval over which the particles on side $i$ may
move:
\begin{equation*}
J_{\mf a, i} \;=\; J_{\mf a, i} (\mb k, \bs \ell)\;=\; 
\Big\{ 1 - \mb 1\{\ell_{i-1} = n^{\mf a}_{i-1}\} \,,\, 
\dots \,,\, n^{\mf a}_i + \mb 1\{k_{i+1}\le 1\} \Big\}\;,
\end{equation*}
and by $T^b_{\mf a, i} (\mb k, \bs \ell)$, $b, b+1 \in J_{\mf a, i}$,
the vector obtained from $(\mb k, \bs \ell)$ by replacing the occupied
sites $k_i$, $k_i+1$ by the sites $b$, $b+1$. Note that $T^b_{\mf a,
  i} (\mb k, \bs \ell)$ belongs to $I^*_{\mf a}$ because we assumed
$n>3$. Note also that we did not excluded the possibility that $b=k_i$
in which case $T^b_{\mf a, i} (\mb k, \bs \ell) = (\mb k, \bs \ell)$.

Denote by $\mc N(\eta)$ the set of all configurations which can be
attained from $\eta\in\Omega^2$ by a rate $e^{-\beta}$ jump.  Note
that the set $\mc N(\eta^{\mf a, (\mb k, \bs \ell)}_{\mb x})$ may have
more than $8$ configurations. For example, if $\mf a=\mf s$, $\mb x =
\mb w$, $\ell_0=n-3$ and $k_1\ge 2$, the particle at $(n-2,1)$ jumps
at rate $e^{-\beta}$ to $(n-2,0)$. However, starting from this
configuration, the probability of the event $H(\bb H_{01}) \not =
H(\eta_{\mb w}^{\mf s, (\mb k,\bs \ell)})$ converges to $0$ since the
unique rate one jump from this configuration is the return to
$\eta_{\mb w}^{\mf s, (\mb k,\bs \ell)}$.

The proof of the next result is straightforward and left to the
reader. One just needs to identify all configurations which can be
reached by rate $1$ jumps from a configuration in $\mc N(\eta)$.

\begin{lemma}
\label{s13}
Fix $\eta\in\Omega^2$. Then, for all $\xi\in \mc N(\eta)$, there
exists a probability measure $\bb M (\xi, \,\cdot\,)$ defined on $\bb
H_{01}$ such that
\begin{equation*}
\lim_{\beta\to\infty} \mb P^\beta_{\xi}
\big[ \eta (H (\bb H_{01})) \in \Pi \big] \;=\; \bb M (\xi,\Pi)\;,
\quad \Pi \subset \bb H_{01} \;.
\end{equation*}
Moreover, for $0\le i\le 3$, $\mf a\in \{\mf s, \mf l\}$, $(\mb k, \bs
\ell)\in I_{\mf a,i}^\pm$
\begin{equation*}
\begin{split}
& \bb M (\hat T^\pm_{\mf a,i} \eta_{\mb x}^{\mf a, (\mb k,\bs \ell)},
\eta_{\mb x}^{\mf a, T^\pm_{\mf a,i} (\mb k,\bs  \ell)}) 
\;=\; \frac 1{M_i(\mb k,\bs \ell)} \;, \\
& \quad
\bb M (\hat T^\pm_{\mf a,i} \eta_{\mb x}^{\mf a, (\mb k,\bs \ell)},
\eta_{\mb x}^{\mf a, (\mb k,\bs \ell)}) \;=\; 
\frac{M_i(\mb k,\bs \ell)-1}{M_i(\mb k,\bs \ell)} 
\end{split}
\end{equation*}
if $M_i(\mb k,\bs \ell) \ge 3$; and
\begin{equation*}
\begin{split}
& \bb M (\hat T^-_{\mf a,i} \eta_{\mb x}^{\mf a, (\mb k,\bs \ell)},
\eta_{\mb x}^{\mf a, T^b_{\mf a, i} (\mb k, \bs \ell)}) 
\;=\; \mf m (J_{\mf a, i} (\mb k, \bs \ell), k_i-1,b) \; , \quad
b \;,\; b+1\in J_{\mf a, i} (\mb k, \bs \ell)\;, \\
& \quad \bb M (\hat T^+_{\mf a,i} \eta_{\mb x}^{\mf a, (\mb k,\bs \ell)},
\eta_{\mb x}^{\mf a, T^b_{\mf a, i} (\mb k, \bs \ell)}) 
\;=\; \mf m (J_{\mf a, i} (\mb k, \bs \ell), k_i,b) \; , \quad
b \;,\; b+1\in J_{\mf a, i} (\mb k, \bs \ell)\;,
\end{split}
\end{equation*}
if $M_i(\mb k,\bs \ell) =2$, where the probability $\mf m(J, a, c)$
has been introduced in \eqref{b13}.
\end{lemma}

Let
\begin{equation*}
Z(\eta_{\mb x}^{\mf a, (\mb k,\bs \ell)}) \; =\; 
\sum_{\xi \in \mc N(\eta_{\mb x}^{\mf a, (\mb k,\bs \ell)})}
\sum_{\zeta\not = \eta_{\mb x}^{\mf a, (\mb k,\bs \ell)}}
\bb M(\xi,\zeta) \; =\; \sum_{\xi \in \mc N(\eta_{\mb x}^{\mf a, (\mb k,\bs \ell)})}
\{1 - \bb M(\xi, \eta_{\mb x}^{\mf a, (\mb k,\bs \ell)})\} \;.
\end{equation*}
Note that $Z(\eta_{\mb x}^{\mf a, (\mb k,\bs \ell)})$ does not depend
on $\mb x$ and that 
\begin{equation*}
\begin{split}
Z(\eta_{\mb w}^{\mf a, (\mb k,\bs \ell)}) \; & =\; \sum_{i=0}^3 
\frac{\mb 1\{M_i(\mb k,\bs \ell) > 2\}} {M_i(\mb k,\bs \ell)} 
\, \Big\{ \, \mb 1\{(\mb k,\bs \ell) \in I^-_{\mf a,i}\} + 
\mb 1\{(\mb k,\bs \ell) \in I^+_{\mf a,i}\} \Big\} \\
& +\; \sum_{i=0}^3 \mb 1\{M_i(\mb k,\bs \ell) = 2\}  
\, \mb 1\{(\mb k,\bs \ell) \in I^-_{\mf a,i}\} \, 
\big [1- \mf m(J_{\mf a, i} , k_i-1, k_i)\big ]  \\
& +\; \sum_{i=0}^3 \mb 1\{M_i(\mb k,\bs \ell) = 2\}  
\, \mb 1\{ (\mb k,\bs \ell) \in I^+_{\mf a,i} \} \, 
\big [1-\mf m(J_{\mf a, i} , k_i, k_i)\big ] \;.   
\end{split}
\end{equation*}

\begin{proposition}
\label{s06}
Fix $\mb x\in \Lambda_L$, $\mf a\in\{\mf s, \mf l\}$, $(\mb k, \bs
\ell) \in I^*_{\mf a}$. Then,
\begin{enumerate}
\item The triple $(\{\eta_{\mb x}^{\mf a, (\mb k,\bs \ell)}\},
  \{\eta_{\mb x}^{\mf a, (\mb k,\bs \ell)}\}\cup \Delta_1, \eta_{\mb
    x}^{\mf a, (\mb k,\bs \ell)})$ is a valley of depth given by
  $\mu_K (\eta_{\mb x}^{\mf a, (\mb k,\bs \ell)})/$
  $\Cap_K(\{\eta_{\mb x}^{\mf a, (\mb k,\bs \ell)}\}, [\{\eta_{\mb
    x}^{\mf a, (\mb k,\bs \ell)}\}\cup \Delta_1]^c)$;

\item Under $\mb P^\beta_{\eta_{\mb x}^{\mf a, (\mb k,\bs \ell)}}$, $H
  (\bb H_{01}\setminus \{\eta_{\mb x}^{\mf a, (\mb k,\bs
    \ell)}\})/e^\beta$ converges in distribution to an exponential
  random variable of parameter $Z(\eta_{\mb x}^{\mf a, (\mb k,\bs
    \ell)})$;

\item For any $\Pi \subset \bb H_{01}\setminus\{\eta_{\mb x}^{\mf a, (\mb
    k,\bs \ell)}\}$,
\begin{equation*}
\begin{split}
& \lim_{\beta\to\infty} \mb P^\beta_{\eta_{\mb x}^{\mf a, (\mb k,\bs
    \ell)}}
\big[ \eta (H (\bb H_{01}\setminus\{\eta_{\mb x}^{\mf a, (\mb
    k,\bs \ell)}\})) \in \Pi \big] \\
&\qquad \;=\; \frac 1{Z(\eta_{\mb x}^{\mf a, (\mb k,\bs \ell)})} 
\sum_{\xi \in \mc N(\eta_{\mb x}^{\mf a, (\mb k,\bs \ell)})} 
\bb M (\xi,\Pi)\;=:\; Q(\eta_{\mb x}^{\mf a, (\mb k,\bs \ell)}, \Pi)\;.
\end{split}
\end{equation*}
\end{enumerate}
\end{proposition}

\begin{proof}
Recall \cite[Theorem 2.6]{bl2}. Assumption (2.14) is fulfilled by
default and assumption (2.15) follows from the definition of the set
$\Delta_1$. This proves the first assertion of the proposition.

The proof of the second claim is simpler than the one of the second
assertion of Proposition \ref{s09} if we take $\tau_1$ as the time of
the first jump. With this definition, $\tau_1$ and
$\eta^\beta_{\tau_1}$ are independent random variables by the Markov
property, $\tau_1/e^\beta$ converges to an exponential random variable
of parameter $|\mc N|$, where $\mc N = \mc N(\eta^{\mf a, (\mb k, \bs
  \ell)}_{\mb x})$, and $\eta^\beta_{\tau_1}$ converges to a random
variable which is uniformly distributed over $\mc N$.

By the arguments of Proposition \ref{s09}, starting from $\eta_{\mb
  x}^{\mf a, (\mb k,\bs \ell)}$, $H (\bb H_{01}\setminus \{\eta_{\mb
  x}^{\mf a, (\mb k,\bs \ell)}\})/e^\beta$ converges in distribution
to an exponential random variable of parameter
\begin{equation*}
\gamma \;=\; \lim_{\beta\to\infty} \sum_{\xi\in \mc N}
\mb P^\beta_{\xi} \Big[ H(\bb H_{01}) \not = H(\eta_{\mb
    x}^{\mf a, (\mb k,\bs \ell)})\Big]\;.
\end{equation*}
To conclude the proof, it remains to recall the statement of Lemma
\ref{s13}, and the definition of $Z(\eta_{\mb x}^{\mf a, (\mb k,\bs
  \ell)})$.

The proof of the third assertion of the proposition is identical to the
one of the third claim of Proposition \ref{s09}.
\end{proof}

As in \eqref{19}, the second assertion of this proposition gives an
explicit expression for the depth of the valley presented in the first
statement. On the other hand, following \eqref{b12}, for
$\eta\in\Omega_2$, let
\begin{equation*}
\mb R(\eta, \Pi) \;=\; 
Z(\eta)\, Q(\eta, \Pi) \;=\; 
\sum_{\xi \in \mc N(\eta)} 
\bb M (\xi,\Pi) \;, \quad
\Pi\subset \bb H_{01}\;. 
\end{equation*}
By Lemma \ref{s13}, if $M_i(\mb k,\bs \ell) >2$ for some $0\le i\le
3$,
\begin{equation*}
\mb R (\eta^{\mf a, (\mb k,\bs \ell)}_{\mb x},
\eta^{\mf a, T^\pm_{\mf a, i}(\mb k,\bs \ell)}_{\mb x}) \;= \; 
\frac {\mb  1\{(\mb k,\bs \ell) \in I^\pm_{\mf a,i}\}} 
{M_i(\mb k,\bs \ell)}
\end{equation*}
and if $M_i(\mb k,\bs \ell) =2$,
\begin{equation*}
\mb R (\eta^{\mf a, (\mb k,\bs \ell)}_{\mb x},
\eta^{\mf a, T^b_{\mf a, i}(\mb k,\bs \ell)}_{\mb x}) \;=\;
p_{\mf a}(\bs k, \bs \ell, i, b)\;,  
\quad b, b+1 \in J_{\mf a, i} (\mb k, \bs \ell)\;,
\end{equation*}
where
\begin{equation*}
p_{\mf a}(\bs k, \bs \ell, i, b)  
\;=\; \mb 1\{(\mb k,\bs \ell) \in I^-_{\mf a,i}\} \,
\mf m (J_{\mf a, i}, k_i-1, b) \;+\;
\mb 1\{(\mb k,\bs \ell) \in I^+_{\mf a,i}\} \, \mf m 
(J_{\mf a, i}, k_i, b)\;.
\end{equation*}

It follows from Proposition \ref{s06} and Lemma \ref{s13} that
starting from a configuration $\zeta\in \Omega^2$ the process
$\eta^\beta_t$ reaches $\bb H_{01}$ only in a configuration of
$\Omega^1 \cup \Omega^2$ or in a configuration in which all sites of a
$(n-1)\times (n+1)$ rectangle are occupied and an extra particle is
attached to a side of length $n+1$. To pursue our analysis, we have to
investigate this new set of configurations.
 
\subsection{The valleys $\mc E^{\mf a, i}_{\mb x}$}
\label{ss3}

The arguments of this subsection are similar to the ones of Subsection
\ref{ss1}.  Let $T^{\mf l}$, $T^{\mf s}$ be the rectangles $T^{\mf l}
=\{0,\dots, n\}\times\{0, \dots, n-2\}$, $T^{\mf s} = \{0, \dots,
n-2\} \times \{0,\dots, n\}$. Denote by $T^{\mf a}_{\mb x}$, $\mf a
\in\{ \mf s, \mf l\}$, $\mb x\in\Lambda_L$, the rectangle $T^{\mf a}$
translated by $\mb x$: $T^{\mf a}_{\mb x} = {\mb x} + T^{\mf a}$, and
by $\eta_{\mb x, \mf a}$ the configuration in which all sites of
$T^{\mf a}_{\mb x}$ are occupied. Note that $\eta_{\mb x, \mf a}$
belongs to $\Omega_{L,K-1}$ and not to $\Omega_{L,K}$.

For $\mf a \in\{ \mf s, \mf l\}$, $\mb x\in\Lambda_L$, $\mb z\in
\partial_+ T^{\mf a}_{\mb x}$, denote by $\eta_{\mb x, \mf a}^{\mb z}$
the configuration in which all sites of the rectangle $T^{\mf a}_{\mb
  x}$ and the site $\mb z$ are occupied: $\eta_{\mb x, \mf a}^{\mb z}
= \eta_{\mb x, \mf a} + \mf d_z$, where $\mf d_y$, $y\in \Lambda_L$,
is the configuration with a unique particle at $y$ and summation of
configurations is performed componentwise.  Denote by $\partial_j
T^{\mf a}_{\mb x}$, $0\le j\le 3$, the $j$-th boundary of $T^{\mf
  a}_{\mb x}$:
\begin{equation*}
\begin{split}
& \partial_jT^{\mf a}_{\mb x}  \;=\; \{\mb z\in \partial_+ T^{\mf a}_{\mb x} : 
\exists\, \mb y\in  T^{\mf a}_{\mb x}\,;\, \mb y -\mb z = (1-j) e_2\} 
\quad j=0,2 \; , \\
& \quad \partial_j T^{\mf a}_{\mb x} \;=\; 
\{\mb z\in \partial_+ T^{\mf a}_{\mb x} : 
\exists\, \mb y\in T^{\mf a}_{\mb x} \,;\, \mb y -\mb z = (j-2) e_1\} 
\quad j=1,3\;.
\end{split}
\end{equation*} 
Let 
\begin{equation*}
\mc  E^{\mf a,j}_{\mb x} \;=\; \{ \eta_{\mb x, \mf a}^{\mb z} : 
\mb z\in \partial_j T^{\mf a}_{\mb x}\}\;,
\end{equation*}
and let $\Omega^3= \Omega^3_{L,K}$ be the set of all such
configurations:
\begin{equation*}
\Omega^3_{\mb x} \;=\; \bigcup_{j=0}^3 \bigcup_{\mf a \in 
\{ \mf s, \mf l\}} \mc  E^{\mf a,j}_{\mb x} \;,
\quad
\Omega^3 \;=\; \bigcup_{\mb x\in\Lambda_L} \Omega^3_{\mb x} \;.
\end{equation*} 

\begin{figure}[!h]
\begin{center}
\includegraphics[scale =0.5]{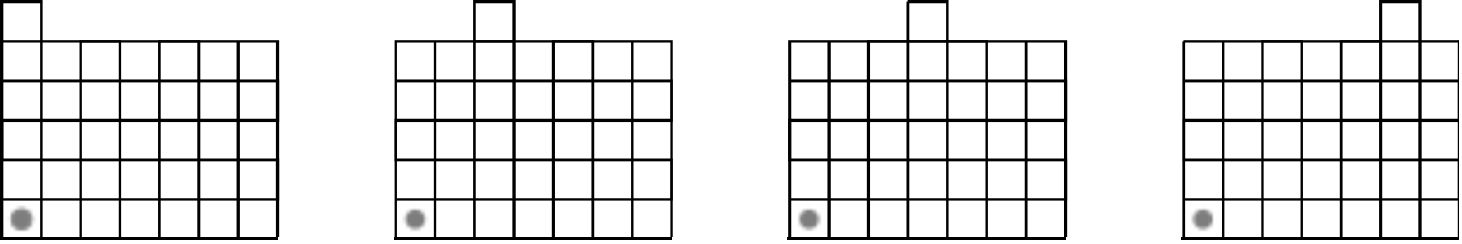}
\end{center}
\caption{Four among the seven configurations of $\mc E^{\mf l,2}_{\mb
    x}$ for $n=6$. The gray dot represents the site $\mb x$.}
\end{figure}

The process $\{\eta^\beta_t : t\ge 0\}$ can reach any configuration
$\xi\in \mc E^{\mf a,j}_{\mb x}$ from any configuration $\eta\in \mc
E^{\mf a,j}_{\mb x}$ with rate one jumps.  The main result of this
subsection states that for any configuration $\xi\in \mc E^{\mf
  a,j}_{\mb x}$, the triples $(\mc E^{\mf a,j}_{\mb x}, \mc E^{\mf
  a,j}_{\mb x}\cup \Delta_1, \xi)$ are valleys.

Denote by $\mc N(\mc E^{\mf a,j}_{\mb x})$, the configurations which
do not belong to $\mc E^{\mf a,j}_{\mb x}$, but which can be reached
from a configuration in $\mc E^{\mf a,j}_{\mb x}$ by performing a jump
of rate $e^{-\beta}$. The set $\mc N(\mc E^{\mf s,0}_{\mb w})$ has the
following $n+3$ elements. There are $n+1$ configurations obtained when
the bottom particle detaches itself from the others: $\eta_{\mb w, \mf
  s} + \mf d_z$, where $\mb z\in J_2 =\{(-1,-1), (a,-2), (n-1,-1) :
0\le a\le n-2\}$. There is a configuration in $\mc N(\mc E^{\mf
  s,0}_{\mb w})$ which is obtained when the bottom particle is at
$(1,-1)$ and the particle at $\mb w$ moves to $\mb w-e_2$:
$\sigma^{\mb w, \mb w-e_2} \eta_{\mb w, \mf s}^{(1,-1)}$. The last
configuration of $\mc N(\mc E^{\mf s,0}_{\mb w})$ is obtained when the
bottom particle is at $(n-3,-1)$ and the particle at $\mb w_1-e_1$
moves to $\mb w_1-e_1-e_2$: $\sigma^{\mb w_1-e_1, \mb w_1-e_1-e_2}
\eta_{\mb w, \mf s}^{(n-3,-1)}$.

\begin{lemma}
\label{s07}
Fix $\mb x\in \Lambda_L$, $\mf a\in\{\mf s, \mf l\}$ and $0\le j\le
3$. For each $\xi\in \mc N(\mc E^{\mf a,j}_{\mb x})$, there exists a
probability measure $\bb M (\xi, \,\cdot\,)$ defined on $\bb H_{01}$
such that
\begin{equation*}
\lim_{\beta\to\infty} \mb P^\beta_{\xi} 
\big[ \eta(H_{\bb H_{01}}) \in \Pi \big] \;=\; \bb M (\xi,
\Pi)\;, \quad \Pi \subset\bb H_{01}\;.
\end{equation*}
Moreover, if $\xi\in \mc N(\mc E^{\mf s,0}_{\mb w})$,
\begin{equation*}
\bb M (\eta_{\mb w, \mf s} + \mf d_z , \mc E^{\mf s,j}_{\mb w})
\;=\;
\mf p(\mb z, \partial_j T^{\mf s}_{\mb w}, \partial_+ T^{\mf s}_{\mb w})
\;,\quad 0\le j\le 3\;,\;\;\mb z\in J_2 \;.
\end{equation*}
\begin{equation*}
\bb M (\sigma^{\mb w, \mb w-e_2} \eta_{\mb w, \mf s}^{(1,-1)} , 
\Pi) \;=\;
\left\{
\begin{split}
& \frac 1{n+1} \quad \text{\rm if}\quad \Pi = 
\{\sigma^{\mb w_3+e_2, \mb w-e_2} \eta_{\mb w, \mf s}^{(1,-1)}\}\;, \\
& \frac n{n+1} \quad \text{\rm if}\quad \Pi = 
\{\eta_{\mb w, \mf s}^{(1,-1)}\} \;.
\end{split}
\right.
\end{equation*}
\begin{equation*}
\bb M (\xi , \Pi) \;=\;
\left\{
\begin{split}
& \frac 1{n+1} \quad \text{\rm if}\quad \Pi = 
\{\sigma^{\mb w_2-e_1+e_2, \mb w_1-e_1-e_2} \eta_{\mb w, \mf s}^{(n-3,-1)}\}\;, \\
& \frac n{n+1} \quad \text{\rm if}\quad \Pi = 
\{\eta_{\mb w, \mf s}^{(n-3,-1)}\} \;,
\end{split}
\right.
\end{equation*}
if $\xi = \sigma^{\mb w_1-e_1, \mb w_1-e_1-e_2} \eta_{\mb w, \mf s}^{(n-3,-1)}$.
\end{lemma}

The proof of the previous lemma is simpler than the one of Lemma
\ref{s10} and left to the reader. By symmetry, the distribution of
$\eta(H (\bb H_{01}\setminus\mc E^{\cdot,\cdot}_{\mb x}))$ can be
obtained from the one of $\eta(H (\bb H_{01}\setminus\mc E^{\mf
  s,0}_{\mb w}))$ or from the one of $\eta(H (\bb H_{01}\setminus\mc
E^{\mf s,1}_{\mb w}))$.

Define
\begin{equation*}
Z(\mc E^{\mf a,j}_{\mb x}) \;=\; \sum_{\xi\in \mc N(\mc E^{\mf a,j}_{\mb x})} 
\bb M (\xi, (\mc E^{\mf a,j}_{\mb x})^c) 
\;=\; \sum_{\xi\in \mc N(\mc E^{\mf a,j}_{\mb x})} 
\{ 1 - \bb M (\xi, \mc E^{\mf a,j}_{\mb x})\} \;.
\end{equation*}
By the previous lemma,
\begin{equation*}
Z(\mc E^{\mf s,0}_{\mb w}) 
\;=\; \frac 2{n+1} \;+\; \sum_{\mb y\in J_2} [1- \mf p(\mb y, 
\partial_0 T^{\mf s}_{\mb w}, \partial_+ T^{\mf s}_{\mb w})]\;.
\end{equation*}

\begin{proposition}
\label{s05}
Fix $0\le j\le 3$, $\mf a\in\{\mf l, \mf s\}$, $\mb x\in\Lambda_L$.
\begin{enumerate}
\item For every $\xi\in \mc E^{\mf a, j}_{\mb x}$, the triple $(\mc
  E^{\mf a, j}_{\mb x}, \mc E^{\mf a, j}_{\mb x}\cup \Delta_1, \xi)$
  is a valley of depth $\mu_K (\mc E^{\mf a, j}_{\mb
    x})/$ $\Cap_K(\mc E^{\mf a, j}_{\mb x}, [\mc E^{\mf a, j}_{\mb
    x}\cup \Delta_1]^c)$;

\item For any $\xi\in \mc E^{\mf a, j}_{\mb x}$, under $\mb
  P^\beta_{\xi}$, $H (\bb H_{01}\setminus\mc E^{\mf a, j}_{\mb x})/e^\beta$
  converges in distribution to an exponential random variable of
  parameter $Z(\mc E^{\mf a, j}_{\mb x})/|\mc E^{\mf a, j}_{\mb x}|$;

\item For any $\xi\in \mc E^{\mf a, j}_{\mb x}$, $\Pi \subset
  \bb H_{01}\setminus\mc E^{\mf a, j}_{\mb x}$, 
\begin{equation*}
\lim_{\beta\to\infty} \mb P^\beta_{\xi}
\big[ \eta (H (\bb H_{01}\setminus\mc E^{\mf a, j}_{\mb x})) \in \Pi \big]
\;=\; \frac 1{Z(\mc E^{\mf a, j}_{\mb x})} 
\sum_{\eta\in \mc N(\mc E^{\mf a, j}_{\mb x})} \bb M (\eta, \Pi)
\;=:\; Q(\mc E^{\mf a, j}_{\mb x}, \Pi)\;.
\end{equation*}
\end{enumerate}
\end{proposition}

The proof of this proposition is similar to the one of Proposition
\ref{s09}, with $\tau_1$ defined as the first time the process leaves
the set $\mc E^{\mf a, j}_{\mb x}$. Remark \eqref{19} concerning the
explicit formula for the depth of the valley appearing in the first
statement of Proposition \ref{s05} also holds.

For $\mf a \in\{ \mf s, \mf l\}$, $\mb x\in\Lambda_L$, $0\le j\le 3$,
let $\mc E^{\mf a,j}_{\mb x}$
\begin{equation*}
\mb R(\mc E^{\mf a,j}_{\mb x}, \Pi) \;=\; 
Z(\mc E^{\mf a,j}_{\mb x})\, Q(\mc E^{\mf a,j}_{\mb x}, \Pi)
\;=\; \sum_{\eta\in \mc N(\mc E^{\mf a, j}_{\mb x})} \bb M (\eta, \Pi) 
\;, \quad \Pi\subset \bb H_{01}\;. 
\end{equation*}

It follows from the previous two results that starting from a
configuration $\zeta\in \Omega^3$ the process $\eta^\beta_t$ reaches
$\bb H_{01}$ only in a configuration of $\Omega^2 \cup \Omega^3$ or in
a configuration in which all sites of a $(n-3)\times n$ rectangle are
occupied with $3n$ extra particles attached to the boundary. This is
the last set of configurations which needs to be examined.
 
\subsection{The valleys $\{\zeta_{\mb x}^{\mf a, (\mb
  k,\bs \ell)}\}$}
\label{ss4}

The arguments of this subsection are similar to the ones of Subsection
\ref{ss2}. Let $R^{2,\mf l}$, $R^{2,\mf s}$ be the rectangles
$R^{2,\mf l} =\{1,\dots, n\}\times\{1, \dots, n-3\}$, $R^{2,\mf s} =
\{1, \dots, n-3\} \times \{1,\dots, n\}$. Let $n_0^{2,\mf s} =
n_2^{2,\mf s} = n-3$, $n_1^{2,\mf s} = n_3^{2,\mf s} = n$ be the
length of the sides of the standing rectangle $R^{2,\mf
  s}$. Similarly, denote by $n_i^{2,\mf l}$, $0\le i\le 3$, the length
of the sides of the lying rectangle $R^{2,\mf l}$: $n_i^{2,\mf l} =
n_{i+1}^{2,\mf s}$, where the sum over the index $i$ is performed
modulo $4$.

Denote by $\bb I_{2,\mf a}$, $\mf a \in \{\mf s ,\mf l\}$, the set of
pairs $(\mb k, \bs \ell)$ such that
\begin{itemize}
\item $0\le k_i \le \ell_i \le n_i^{2,\mf a}$,
\item If $k_j=0$, then $\ell_{j-1}= n_{j-1}^{2,\mf a}$.
\end{itemize}
For $(\mb k,\bs \ell) \in \bb I_{2,\mf a}$, $\mf a \in\{ \mf s, \mf
l\}$, let $R^{2,\mf l}(\mb k,\bs \ell)$, $R^{2,\mf s}(\mb k,\bs \ell)$
be the sets
\begin{equation*}
\begin{split}
R^{2,\mf l} (\mb k,\bs \ell)\; & =\;
R^{2,\mf l} \; \cup\; \{(a,0) : k_0\le a\le \ell_0\} \; \cup\;
\{(n+1,b) : k_1\le b\le \ell_1\}\;\cup \\
& \cup\; \{(n+1-a,n-2) : k_2\le a\le \ell_2\}
\;\cup\; \{(0,n-2-b) : k_3\le b\le \ell_3\}\;, \\
R^{2,\mf s} (\mb k,\bs \ell)  \;& =\;
R^{2,\mf s} \; \cup\; \{(a,0) : k_0\le a\le \ell_0\} \; \cup\;
\{(n-2,b) : k_1\le b\le \ell_1\} \;\cup \\
& \cup\; \{(n-2-a,n+1) : k_2\le a\le \ell_2\}
\;\cup\; \{(0,n+1-b) : k_3\le b\le \ell_3\} \;.
\end{split}
\end{equation*}

Denote by $I_{2,\mf a}$, $\mf a \in \{ \mf s, \mf l\}$, the set of
pairs $(\mb k, \bs \ell)\in \bb I_{2,\mf a}$ such that $|R^{2,\mf a}
(\mb k,\bs \ell)|=n^2$. For $(\mb k,\bs \ell) \in I_{2,\mf a}$, denote
by $M^{2,\mf a}_i(\mb k,\bs \ell)$ the number of particles attached to
the side $i$ of the rectangle $R^{2,\mf a} (\mb k,\bs \ell)$:
\begin{equation*}
M^{2,\mf a}_i(\mb k,\bs \ell) \;=\; 
\begin{cases}
\ell_i-k_i+1 & \text{if $k_{i+1}\ge 1\;,$}\\
\ell_i-k_i+2 & \text{if $k_{i+1}=0$}\ .
\end{cases}
\end{equation*}
Clearly, for $(\mb k, \bs \ell)\in I_{2,\mf a}$, $\sum_{0\le i\le 3}
M^{2,\mf a}_i(\mb k,\bs \ell) = 3n + A$, where $A$ is the number of
occupied corners, which are counted twice since they are attached to
two sides.

Denote by $I^*_{2,\mf a} \subset I_{2,\mf a}$, the set of pairs $(\mb
k, \bs \ell)\in I_{2,\mf a}$ whose rectangles $R^{2,\mf a} (\mb k,\bs
\ell)$ have at least two particles on each side: $M^{2,\mf a}_i(\mb
k,\bs \ell)\ge 2$, $0\le i\le 3$. Note that if $(\mb k, \bs \ell)$
belongs to $I^*_{2,\mf a}$, for all $\mb x\in R^{2,\mf a} (\mb k,\bs
\ell)$, there exist $\mb y$, $\mb z \in R^{2,\mf a} (\mb k,\bs \ell)$,
$\mb y\not = \mb z$, with the property $\Vert \mb x-\mb y\Vert =\Vert \mb x-\mb
z\Vert =1$.

For $(\mb k,\bs \ell) \in I_{2,\mf a}$, $\mf a \in\{ \mf s, \mf l\}$,
$\mb x\in \Lambda_L$, let $R^{2,\mf a}_{\mb x} (\mb k,\bs \ell) = \mb x
+ R^{2,\mf a} (\mb k,\bs \ell)$, and let $\zeta_{\mb x}^{\mf a, (\mb
  k,\bs \ell)}$ represent the configurations defined by
\begin{equation*}
\text{$\zeta_{\mb x}^{\mf a, (\mb k,\bs \ell)} (a,b) = 1$
    if and only if $(a,b)\in R_{\mb x}^{2,\mf a}(\mb k,\bs \ell)$}\;.
\end{equation*}
The configurations $\zeta_{\mb x}^{\mf a, (\mb k,\bs \ell)}$, $(\mb
k,\bs \ell) \in I_{2,\mf a}$, have at least four particles attached to
the longer side, and the configurations $\zeta_{\mb x}^{\mf a, (\mb
  k,\bs \ell)}$, $(\mb k,\bs \ell) \in I_{2,\mf a} \setminus
I^*_{2,\mf a}$, belong to $\Omega^3$, forming a $(n-1)\times(n+1)$
rectangle of particles with one extra particle attached to a side of
length $n-1$.  Let $\Omega^4 = \Omega^4_{L,K}$, be the set of
configurations associated to the pairs $(\mb k,\bs \ell)$ in
$I^*_{2,\mf a}$:
\begin{equation*}
\Omega^4_{\mb x} \;=\;  \{ \zeta_{\mb x}^{\mf a, (\mb k,\bs \ell)} :
\mf a \in\{ \mf s, \mf l\} \,,\, (\mb k,\bs \ell) \in
I^*_{2,\mf a} \} \;, \quad
\Omega^4 \;=\;  \bigcup_{\mb x\in\Lambda_L} \Omega^4_{\mb x} \;.
\end{equation*}

\begin{figure}[!h]
\begin{center}
\includegraphics[scale =0.5]{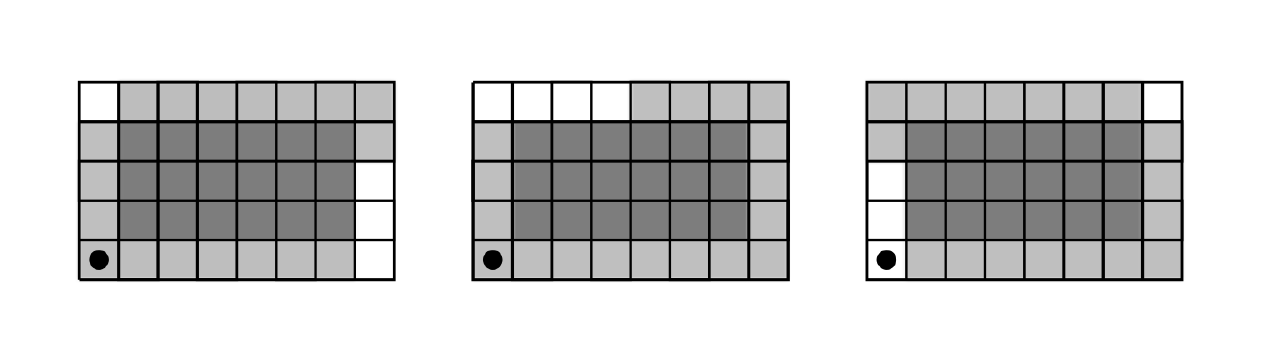}
\end{center}
\caption{Examples of configurations in $\Omega^4_{\mb x}$ for
  $n=6$. In general $3n$ particles (or $n-2$ holes) have to be placed
  around the rectangle, respecting the constraints introduced above.
  The black dot represents the site $\mb x$.}
\end{figure}

We now describe the configurations which can be attained from a
configuration in $\Omega^4$.  Denote by $I_{2,\mf a,i}^\pm$, $\mf a\in
\{\mf s, \mf l\}$, $0\le i\le 3$, the subset of $I_{2,\mf a}^*$
defined by
\begin{equation*}
\begin{split}
& I_{2,\mf a,i}^- \;=\; \{(\mb k, \bs \ell)\in I^*_{2, \mf a} : 
k_i \ge 2 \text{ or } k_i=1 \,,\, \ell_{i-1} = n_{i-1}^{2,\mf a}\} \; , 
\\
&\quad
I_{2,\mf a,i}^+ \;=\; \{(\mb k, \bs \ell)\in I^*_{2, \mf a} : 
\ell_i \le   n_{i}^{2,\mf a}-1 \text{ or } \ell_i =  n_{i}^{2, \mf a}
\,,\, k_{i+1} = 1 \} \;.
\end{split}
\end{equation*}

For $(\mb k, \bs \ell)\in I_{2, \mf a,i}^-$, denote by $\hat
T^-_{2,\mf a,i} \zeta_{\mb x}^{\mf a, (\mb k,\bs \ell)}$ the
configuration obtained from $\zeta_{\mb x}^{\mf a, (\mb k,\bs \ell)}$
by moving the particle sitting at $k_i$ to $k_i-1$. As in Subsection
\ref{ss2}, the abuse of notation is clear. Similarly, for $(\mb k, \bs
\ell)\in I_{2,\mf a,i}^+$, denote by $\hat T^+_{2,\mf a,i} \zeta_{\mb
  x}^{\mf a, (\mb k,\bs \ell)}$ the configuration obtained from
$\zeta_{\mb x}^{\mf a, (\mb k,\bs \ell)}$ by moving the particle
sitting at $\ell_i$ to $\ell_i+1$.

Define the map $T^-_{2,\mf a,i} : I^-_{2,\mf a,i} \to I_{2, \mf a}$ by
\begin{equation*}
T^-_{2,\mf a,i} (\mb k, \bs \ell) \;=
\begin{cases}
(\mb k - \mf e_i , \bs \ell - \mf e_i ) & \text{if $k_{i+1}\ge 1$} \;,
\\
(\mb k - \mf e_i + \mf e_{i+1}, \bs \ell )& 
\text{if $k_{i+1}=0$}\;.
\end{cases}
\end{equation*}
The map $T^+_{2,\mf a,i} : I^+_{2,\mf a,i} \to I_{2,\mf a}$ is defined in an
analogous way.

The vector $T^\pm_{2,\mf a,i}(\mb k, \bs \ell)$ may not belong to
$I^*_{2,\mf a}$ when there are only two particles on one side of a
rectangle $R^{2,\mf a}$ and one of them is translated along another
side. Since there are at least four particles attached to the longer
sides of the rectangle, this may happen only in the shorter sides of
the rectangles. In this case the configuration associated to the
vector $T^\pm_{2,\mf a,i}(\mb k, \bs \ell)$ belongs to $\Omega^3$.

Fix a vector $(\mb k, \bs \ell) \in I^*_{2,\mf a}$ such that $M^{2,\mf
  a}_i(\mb k,\bs \ell)= 2$ for some $0\le i\le 3$. Denote by $J_{2,\mf
  a, i} (\mb k, \bs \ell)$ the interval over which the particles on
side $i$ may move:
\begin{equation*}
J_{2, \mf a, i} \;=\; J_{2, \mf a, i} (\mb k, \bs \ell)\;=\; 
\Big\{ 1 - \mb 1\{\ell_{i-1} = n^{2, \mf a}_{i-1}\} \,,\, 
\dots \,,\, n^{2, \mf a}_i + \mb 1\{k_{i+1}\le 1\} \Big\}\;,
\end{equation*}
and by $T^b_{2, \mf a, i} (\mb k, \bs \ell)$, $b, b+1 \in J_{2, \mf a,
  i}$, the vector obtained from $(\mb k, \bs \ell)$ by replacing the
occupied sites $(k_i, k_i+1)$ by $(b,b+1)$. Note that $T^b_{2,\mf a,
  i} (\mb k, \bs \ell)$ always belongs to $I^*_{2, \mf a}$, and that
we did not excluded the possibility that $b=k_i$ in which case
$T^b_{2, \mf a, i} (\mb k, \bs \ell) = (\mb k, \bs \ell)$.

Denote by $\mc N(\eta)$ the set of all configurations which can be
reached from $\eta\in\Omega^4$ by a rate $e^{-\beta}$ jump.

\begin{lemma}
\label{s12}
Fix $\eta\in\Omega^4$. For each $\xi\in \mc N(\eta)$, there exists a
probability measure $\bb M (\xi, \,\cdot\,)$ defined on $\bb H_{01}$
such that
\begin{equation*}
\lim_{\beta\to\infty} \mb P^\beta_{\xi}
\big[ \eta (H (\bb H_{01})) \in \Pi \big] \;=\; \bb M (\xi,\Pi)\;,
\quad \Pi\subset\bb H_{01}\;.
\end{equation*}
Moreover, for $0\le i\le 3$, $\mf a\in\{\mf s, \mf l\}$ and $(\mb k,
\bs \ell)\in I_{2, \mf a,i}^\pm$,
\begin{equation*}
\begin{split}
& \bb M (\hat T^\pm_{2, \mf a,i} \zeta_{\mb x}^{\mf a, (\mb k,\bs \ell)},
\zeta_{\mb x}^{\mf a, T^\pm_{2,\mf a,i} (\mb k,\bs  \ell)}) 
\;=\; \frac 1{M^{2,\mf a}_i(\mb k,\bs \ell)} \;, \\
& \quad
\bb M (\hat T^\pm_{2,\mf a,i} \zeta_{\mb x}^{\mf a, (\mb k,\bs \ell)},
\zeta_{\mb x}^{\mf a, (\mb k,\bs \ell)}) \;=\; 
\frac{M^{2,\mf a}_i(\mb k,\bs \ell)-1}{M^{2,\mf a}_i(\mb k,\bs \ell)} 
\end{split}
\end{equation*}
if $M^{2,\mf a}_i(\mb k,\bs \ell)\ge 3$; and 
\begin{equation*}
\begin{split}
& \bb M (\hat T^-_{2,\mf a,i} \zeta_{\mb x}^{\mf a, (\mb k,\bs \ell)},
\zeta_{\mb x}^{\mf a, T^b_{2,\mf a, i} (\mb k, \bs \ell)}) 
\;=\; \mf m (J_{2, \mf a, i} (\mb k, \bs \ell), k_i-1,b) \; ,  \\
& \quad \bb M (\hat T^+_{2, \mf a,i} \zeta_{\mb x}^{\mf a, (\mb k,\bs \ell)},
\zeta_{\mb x}^{\mf a, T^b_{2,\mf a, i} (\mb k, \bs \ell)}) 
\;=\; \mf m (J_{2,\mf a, i} (\mb k, \bs \ell), k_i,b) \;, 
\end{split}
\end{equation*}
for $b$, $b+1\in J_{2, \mf a, i} (\mb k, \bs \ell)$ if $M^{2,\mf
  a}_i(\mb k,\bs \ell)=2$. The probability measure $\mf m(J, a, c)$
has been introduced in \eqref{b13}.
\end{lemma}


Let
\begin{equation*}
Z(\zeta_{\mb x}^{\mf a, (\mb k,\bs \ell)}) \; =\; 
\sum_{\xi \in \mc N (\zeta_{\mb x}^{\mf a, (\mb k,\bs \ell)})}
\bb M(\xi, \{\zeta_{\mb x}^{\mf s, (\mb k,\bs \ell)}\}^c) \; =\; 
\sum_{\xi \in \mc N (\zeta_{\mb x}^{\mf a, (\mb k,\bs \ell)})}
\{1 - \bb M(\xi, \zeta_{\mb x}^{\mf a, (\mb k,\bs \ell)})\} \;.
\end{equation*}

\begin{proposition}
\label{s14}
Fix $\mb x\in \Lambda_L$, $\mf a\in \{\mf l, \mf s\}$, $(\mb k,\bs
\ell) \in I^*_{2, \mf a}$. Then,
\begin{enumerate}
\item The triple $(\{\zeta_{\mb x}^{\mf a, (\mb k,\bs \ell)}\},
  \{\zeta_{\mb x}^{\mf a, (\mb k,\bs \ell)}\}\cup \Delta_1, \zeta_{\mb
    x}^{\mf a, (\mb k,\bs \ell)})$ is a valley of depth $\mu_K
  (\zeta_{\mb x}^{\mf a, (\mb k,\bs \ell)})/$ $\Cap_K(\{\zeta_{\mb
    x}^{\mf a, (\mb k,\bs \ell)}\}, [\{\zeta_{\mb x}^{\mf a, (\mb
    k,\bs \ell)}\}\cup \Delta_1]^c)$;

\item Under $\mb P^\beta_{\zeta_{\mb x}^{\mf a, (\mb k,\bs \ell)}}$,
  $H (\bb H_{01}\setminus \{\zeta_{\mb x}^{\mf a, (\mb k,\bs
    \ell)}\})/e^\beta$ converges in distribution to an exponential
  random variable of parameter $Z(\zeta_{\mb x}^{\mf a,
    (\mb k,\bs \ell)})$;

\item For any $\Pi \subset \bb H_{01}\setminus\{\zeta_{\mb x}^{\mf a, (\mb
    k,\bs \ell)}\}$,
\begin{equation*}
\begin{split}
& \lim_{\beta\to\infty} \mb P^\beta_{\zeta_{\mb x}^{\mf a, (\mb k,\bs
    \ell)}}
\big[ \eta (H (\bb H_{01}\setminus\{\zeta_{\mb x}^{\mf a, (\mb
    k,\bs \ell)}\})) \in \Pi \big] \\
&\qquad =\; \frac 1{Z(\zeta_{\mb x}^{\mf a, (\mb k,\bs \ell)})} 
\sum_{\xi \in \mc N(\zeta_{\mb x}^{\mf a, (\mb k,\bs \ell)})} 
\bb M(\xi,\Pi)\;=:\; Q(\zeta_{\mb x}^{\mf a, (\mb  k,\bs \ell)}, \Pi)\;.
\end{split}
\end{equation*}
\end{enumerate}
\end{proposition}

For $\eta\in\Omega_4$, let
\begin{equation*}
\mb R(\eta, \Pi) \;=\; Z(\eta)\, Q(\eta, \Pi) \;=\; 
\sum_{\xi \in \mc N(\eta)} \bb M (\xi,\Pi) \;, \quad
\Pi\subset \bb H_{01}\;. 
\end{equation*}

\section{Tunneling behavior among shallow valleys}
\label{sec4}

We examine in this section the evolution of the Markov process
$\{\eta^\beta_t : t\ge 0\}$ in the time scale $e^\beta$ among the
shallow valleys introduced in the previous section. We first introduce
a family of deep valleys or traps.

\begin{lemma}
\label{s19}
Fix $\mb x\in \Lambda_L$. The triple $(\{\eta^{\mb x}\}, \{\eta^{\mb
  x}\}\cup \Delta_1, \eta^{\mb x})$ is a valley of depth $\mu_K
(\eta^{\mb x})/\Cap_K( \{\eta^{\mb x}\}, [\{\eta^{\mb x}\} \cup
\Delta_1]^c)$.
\end{lemma}

This result follows from \cite[Theorem 2.6]{bl2}. Up to this point, we
introduced five types of disjoint subsets of $\Omega_{L,K}$:
\begin{itemize}
\item $\{\eta^{\mb x}\}$, $\mb x\in\Lambda_L$;
\item $\mc E^{i,j}_{\mb x}$, $0\le i,j\le 3$, $\mb x\in\Lambda_L$;
\item $\{\eta^{\mf a, (\mb k, \bs \ell)}_{\mb x}\}$, $\mb
  x\in\Lambda_L$, $\mf a\in\{\mf l, \mf s\}$, $(\mb k, \bs \ell)\in
  I^*_{\mf a}$;
\item $\mc E^{\mf a, i}_{\mb x}$, $\mf a\in\{\mf l, \mf s\}$, $0\le
  i\le 3$, $\mb x\in\Lambda_L$;
\item $\{\zeta^{\mf a, (\mb k, \bs \ell)}_{\mb x}\}$, $\mb
  x\in\Lambda_L$, $\mf a\in\{\mf l, \mf s\}$, $(\mb k, \bs \ell)\in
  I^*_{2,\mf a}$.
\end{itemize}
Denote by $\mc E_1, \dots, \mc E_\kappa$ an enumeration of these
sets. In this enumeration we shall assume that $\mc E_1 = \{\eta^{\mb
  w}\}$ and that the first $|\Lambda_L|$ sets correspond to the square
configurations: for $1\le i\le |\Lambda_L|$, $\mc E_i = \{\eta^{\mb
  x_i}\}$ for some $\mb x_i\in\Lambda_L$. Some sets $\mc E_j$ are
singletons, as the first $|\Lambda_L|$ sets, and some are not, as the
set $\mc E_{|\Lambda_L|+1} = \mc E^{0,0}_{\mb w}$. Let $\mc E =
\cup_{1\le j\le \kappa} \mc E_j$ be the union of all subsets and let
$\check {\mc E}_j = \cup_{i\not = j} \mc E_i$. For $1\le i\le
|\Lambda_L|$, we sometimes denote $\mc E_i= \{\eta^{\mb x}\}$ by
$\mc E_{\mb x}$.

Let $\Delta^*_1 = \Delta_1 \cup [\bb H_1 \setminus \mc E]$. Fix a
configuration $\xi_i$ in each set $\mc E_i$, $1\le i\le \kappa$. We
proved above and in the previous section that the triples $(\mc E_i,
\mc E_i\cup \Delta_1, \xi_i)$ are valleys. The next result states that
we may increase $\Delta_1$ to $\Delta^*_1$.

\begin{lemma}
\label{s17}
The triples $(\mc E_i, \mc E_i \cup \Delta^*_1, \xi_i)$,
$|\Lambda_L|<i\le \kappa$, are valleys of depth $e^\beta |\mc
E_i|/Z(\mc E_i)$. Moreover, for every $|\Lambda_L|< i \le \kappa$,
$1\le j \not = i \le \kappa$, $\xi\in \mc E_i$,
\begin{equation*}
\lim_{\beta\to\infty} \mb P^\beta_\xi \big[ H(\check{\mc E}_i) 
= H(\mc E_j) \big] \;=\; Q(\mc E_i,\mc E_j)\;.
\end{equation*}
\end{lemma}

\begin{proof}
As already remarked in \eqref{19}, it follows from the second
assertion of the propositions stated in the previous section that the
depth of the valleys $(\mc E_i, \mc E_i \cup \Delta_1, \xi_i)$,
$|\Lambda_L|<i\le \kappa$, is $e^\beta |\mc E_i|/Z(\mc E_i)$.  The
first assertion of the lemma follows from Lemma \ref{s18} below and
from the fact, proved in the previous section, that for
$|\Lambda_L|<i\le \kappa$,
\begin{equation*}
\lim_{\beta\to\infty} \min_{\xi\in\mc E_i}
\mb P^\beta_\xi \big[ H(\bb H_{01} \setminus \mc E_i) = H(\mc E) \big]
\;=\; 1\;.
\end{equation*}

The second statement of the lemma follows from the definition of the
probability measure $Q(\mc E_i, \,\cdot\,)$ introduced in the previous
section. 
\end{proof}

Denote by $\{\eta^{\mc E}_t:t\ge 0\}$ the trace of the process
$\eta^\beta_t$ on $\mc E$. The jumps rates of the Markov process
$\eta^{\mc E}_t$ are represented by $R^{\mc E}_\beta(\eta, \xi)$.
Recall that $\xi_i$ is a fixed configuration in the set $\mc E_i$.

\begin{proposition}
\label{s16}
The sequence of Markov processes $\{\eta^\beta_t : t\ge 0\}$ exhibits
a tunneling behavior on the time-scale $e^\beta$, with metastable states
$\{\mc E_j : 1\le j\le \kappa\}$, metastable points $\xi_j$, $1\le j\le
\kappa$, and asymptotic Markov dynamics characterized by the rates
\begin{equation*}
\begin{split}
& r(\mc E_i, {\mc E}_j) \;=\; 0 \;,\quad 1\le i\le |\Lambda_L|
\;, \; 1\le j \not = i \le \kappa \;, \\
&\quad 
r(\mc E_i, \mc E_j) \;=\; \mb R(\mc E_i, \mc E_j)\;, \quad
|\Lambda_L|< i \le \kappa\;, \; 1\le j \not = i \le \kappa \;.    
\end{split}
\end{equation*}
\end{proposition}

\begin{proof}
We check that the first two assumptions of \cite[Theorem 2.7]{bl2} are
fulfilled. We start with assumption (H1). For the valleys $\mc E_j$
which are singletons, there is nothing to prove. For the other ones,
as $\check {\mc E}_j \subset [\mc E_j \cup \Delta_1]^c$, assumption
(H1) follows from the proofs of Propositions \ref{s09} and \ref{s05}.

We turn to assumption (H0). Denote by $r_\beta(\mc E_i, \mc E_j)$ the
average rates of the trace process:
\begin{equation*}
r_\beta(\mc E_i, \mc E_j)\;=\; \frac 1{\mu_K(\mc E_i)} 
\sum_{\eta\in\mc  E_i} \mu_K(\eta) \sum_{\xi\in\mc  E_j} R^{\mc E}_\beta
(\eta, \xi)\;.
\end{equation*}
We claim that $e^\beta r_\beta(\mc E_i, \mc E_j)$, $1\le i\not = j\le
\kappa$, converges to a limit denoted by $r(i,j)$, and that
$\sum_{j\not =i} r(i,j)=0$, $1\le i\le |\Lambda_L|$, $\sum_{j\not =i}
r(i,j)\in (0,\infty)$, $i> |\Lambda_L|$.

Consider first the case $i> |\Lambda_L|$. We may rewrite $e^\beta
r_\beta(\mc E_i, \mc E_j)$ as $e^\beta r_\beta(\mc E_i, \check{\mc
  E}_i) \times [r_\beta(\mc E_i, \mc E_j)/ r_\beta(\mc E_i, \check{\mc
  E}_i)]$. By \cite[Corollary 4.4]{bl4}, $r_\beta(\mc E_i, \mc E_j)/
r_\beta(\mc E_i, \check{\mc E}_i)$ converges to a number $p(\mc E_i,
\mc E_j)\in [0,1]$.

On the other hand, by \cite[Lemma 6.7]{bl2}, $e^\beta r_\beta(\mc E_i,
\check{\mc E}_i) = e^\beta \Cap_K (\mc E_i, \check{\mc
  E}_i)/\mu_K (\mc E_i)$. From the results stated in the previous
section, it is easy to construct a path $\gamma$ from $\mc E_i$ to
$\check{\mc E}_i$ such that $G_K(\gamma) = e^{-\beta}
\mu_K(\eta)$, $\eta\in\mc E_i$. It is also easy to see that any
path $\gamma'$ from $\mc E_i$, to $\check{\mc E}_i$ is such that
$G_K(\gamma) \le e^{-\beta} \mu_K(\eta)$, $\eta\in\mc
E_i$. Hence, $G_K(\mc E_i, \check{\mc E}_i) = e^{-\beta}
\mu_K(\eta)$, $\eta\in\mc E_i$. Assumption (H0) for $i>
|\Lambda_L|$ follows from this identity and \eqref{17}.

Fix now $i\le |\Lambda_L|$. Since $r_\beta(\mc E_i, \mc E_j) \le
r_\beta(\mc E_i, \check{\mc E}_i)$, we have to show that the rescaled
rate $e^\beta r_\beta(\mc E_i, \check{\mc E}_i) = e^\beta \Cap_K
(\mc E_i, \check{\mc E}_i)/\mu_K (\mc E_i)$ vanishes as
$\beta\uparrow\infty$. Since $G_K(\mc E_i, \check{\mc E}_i) \le
e^{-2\beta} \mu_K(\eta)$, $\eta\in\mc E_i$, the result follows
from \eqref{17}.

In view of the proof of \cite[Lemma 10.2]{bl4} and Lemma \ref{s17},
$e^\beta r_\beta(\mc E_i, \mc E_j)$, $1\le i\not = j\le \kappa$,
converges to $Z(\mc E_i) Q(\mc E_i, \mc E_j) = \mb R(\mc E_i, \mc
E_j)$.

It remains to show property (M3) of tunneling, which states that the
time spent outside $\mc E$ is negligible. Fix $1\le i\le \kappa$ and
$\xi\in\mc E_i$. Denote by $\{H_j : j\ge 1\}$ the times of the
successive returns to $\mc E$: $H_1 = H^+(\mc E)$, $H_{j+1} = H^+(\mc
E) \circ \theta_{H_j}$, $j\ge 1$. To prove (M3), it is enough to show
that for all $t>0$ 
\begin{equation}
\label{11}
\begin{split}
& \lim_{k\to\infty} \lim_{\beta\to\infty} \mb P^\beta_\xi 
\big[ H_k \le t e^\beta\big] \;=\;0 \quad\text{and}\\
& \lim_{\beta\to\infty} \mb E^\beta_\xi \big[ e^{-\beta} \int_0^{H_k
  \wedge t e^\beta} \mb 1\{\eta^\beta_{s} \in \Delta^*_1\} 
\, ds \big] \;=\;0 \quad \text{for all $k\ge 1$}\;.
\end{split}
\end{equation}

Since $H_1= H^+(\mc E)$ is greater than the time of the first jump,
there exists a positive constant $c_0$, independent of $\beta$, which
turns $H_1= H^+(\mc E)$ bounded below by a mean $c_0 e^{\beta}$
exponential time, $\mb P^\beta_\eta$ almost surely for all $\eta\in
\mc E$. The first result of \eqref{11} follows from this observation
and of the strong Markov property.

To estimate the second term of of \eqref{11}, fix $k\ge 1$ and rewrite
the time integral as $\sum_{0\le j<k} \int_{H_j \wedge t
  e^\beta}^{H_{j+1} \wedge t e^\beta}$. For a fixed $j$, the integral
vanishes unless $H_j < t e^\beta$. In this case, we may apply the
strong Markov property to estimate the expectation by
\begin{equation*}
k\, \sup_{\xi\in \mc E} \mb E^\beta_\xi \big[ e^{-\beta} \int_0^{H_1
  \wedge t e^\beta} \mb 1\{\eta^\beta_{s} \in \Delta^*_1\} 
\, ds \big] \;.
\end{equation*}
If $\xi$ belongs to $\mc E_i$, $1\le i\le |\Lambda_L|$, the
expectation is bounded above by $t \mb P^\beta_\xi [\tau_1 \le t
e^\beta]$, where $\tau_1$ is the time of the first jump. This
expression vanishes because $\tau_1$ is an exponential time whose mean
is of order $e^{2\beta}$. For $i> |\Lambda_L|$, we have seen in the
proofs of the propositions of the previous section that the time spent
between two visits to $\mc E$ can be estimated by the time a rate $1$,
finite state, irreducible Markov process needs to visit a specific
set. This concludes the proof of the proposition.
\end{proof}

Let $\Psi : \mc E \to \{1, \dots, \kappa\}$ be the index function
$\Psi (\eta) = \sum_{1\le j\le \kappa} j\, \mb 1\{\eta\in\mc E_j\}$.
It follows from the previous result that the non-Markovian process
$X^\beta_t = \Psi(\eta^{\mc E}_{te^\beta})$ converges to the Markov
process on $\{1, \dots, \kappa\}$ with jump rates $r(i,j) = \mb R(\mc
E_i, \mc E_j)$. The states $\{1, \dots, |\Lambda_L|\}$ are absorbing,
while the states $\{|\Lambda_L| +1, \dots, \kappa\}$ are transient for
the asymptotic dynamics.

Let $q(i,j)$, $1\le i\le \kappa$, $1\le j\le |\Lambda_L|$, be the
probability that starting from $i$ the asymptotic process eventually
reaches the absorbing point $j$:
\begin{equation}
\label{10}
q(i,j) \;=\; \bb P_i \big[ X_t = j \text{ for some }
t>0\big]\;, 
\end{equation}
where $\bb P_i$ stands for the probability on the path space
$D([0,\infty), \{1, \dots, \kappa\})$ induced by the Markov process
with rates $r(j,k)$ starting from $i$. We sometimes denote $q(i,j)$ by
$q(\mc E_i, \mc E_j)$.

\section{Tunneling among the deep valleys}
\label{sec2}

We prove in this section the main result of this article.  Recall that
we denoted by $\mc E_{\mb x}$, $\mb x\in \Lambda_L$, the singletons
$\{\eta^{\mb x}\}$, and that we denoted by $\mc N(\eta^{\mb x})$ the set
of configurations which can be reached from $\eta^{\mb x}$ by a jump
of rate $e^{-2\beta}$. By Lemma \ref{bs01}, for each $\xi\in \mc N(
\eta^{\mb x})$ there exists a probability measure $\bb M (\xi,
\,\cdot\,)$ defined on $\bb H_{01}$ such that
\begin{equation*}
\lim_{\beta\to\infty} \mb P^\beta_{\xi} \big[ \eta(H_{\bb H_{01}})\in 
\Pi \big] \;=\; \bb M (\xi, \Pi)\;, \quad \Pi\subset \bb H_{01} \;.
\end{equation*} 

Recall from \eqref{10} the definition of the probability $q(\mc E_j,
\,\cdot\,)$. Let
\begin{equation*}
Z \;=\; 
\sum_{\xi\in \mc N (\eta^{\mb x})} \sum_{j=1}^{\kappa} \bb M (\xi, \mc E_j) 
\, q(\mc E_j, \breve{\mc F}_{\mb x})\;=\; 
\sum_{\xi\in \mc N (\eta^{\mb x})} \sum_{j=1}^{\kappa} 
\bb M (\xi, \mc E_j) \, \big[ 1 - q(\mc E_j, \mc E_{\mb x})\big] \;,
\end{equation*}
where $\breve{\mc F}_{\mb x} = \cup_{\mb y \not = \mb x} \mc E_{\mb
  y}$, the union being carried over $\mb y\in\Lambda_L$.  Recall that
we denote by $\Delta_0$ the configurations which are not ground
states: $\Delta_0 = \Omega_{L,K} \setminus \Omega^0$, and let ${\mc F}
= \cup_{\mb y\in\Lambda_L} \mc E_{\mb y}$.

\begin{proposition}
\label{s02}
Fix $\mb x\in \Lambda_L$.
\begin{enumerate}
\item The triple $(\mc E_{\mb x}, \mc E_{\mb x}\cup \Delta_0, \eta^{\mb
    x})$ is a valley of depth $\mu_K (\eta^{\mb x})/\Cap_K(\mc
  E_{\mb x}, \breve{\mc F}_{\mb x})$;

\item Under $\mb P^\beta_{\eta^{\mb x}}$, $H (\breve{\mc
    F}_{\mb x})/e^{2\beta}$ converges in distribution to an exponential
  random variable of parameter $Z$;

\item For any $\mb y \not = \mb x$,
\begin{equation*}
\lim_{\beta\to\infty} \mb P^\beta_{\eta^{\mb x}} 
\big[ \eta(H(\breve{\mc F}_{\mb x})) = \eta^{\mb y} \big] 
\;=\; \frac 1Z\, \sum_{\xi\in \mc N (\eta^{\mb x})} 
\sum_{j=1}^\kappa \bb M (\xi, \mc E_j) 
\, q(\mc E_j, {\mc E}_{\mb y})\;=: \bb Q(\mb x, \mb y)\;.
\end{equation*}
\end{enumerate}
\end{proposition}

\begin{proof}
Recall \cite[Theorem 2.4]{bl2}. By definition of the set $\Delta_0$,
$\mu_K(\Delta_0)/\mu_K(\mc E_{\mb x})$ is of order
$e^{-\beta}$. Condition (2.15) is therefore fulfilled. Since $\mc
E_{i}$ is a singleton, condition (2.14) holds automatically and the
result follows.

The proof of the second assertion is similar to the one of the
second claim in Proposition \ref{s09} with the following
modifications. We first need to replace the normalization $e^{\beta}$ by
$e^{2\beta}$ and to define $\tau_1$ as the time of the first jump, to
write
\begin{equation*}
H(\breve{\mc F}_{\mb x}) \;=\; \tau_1 \;+\; H({\mc F})\circ 
\theta_{\tau_1} \;+\; 
\mb 1\{ H({\mc F}) \circ \theta_{\tau_1} = H(\mc E_{\mb x}) 
\circ \theta_{\tau_1}\} H(\breve{\mc F}_{\mb x}) \circ
\theta_{H^+({\mc F})}\; .
\end{equation*}
At this point, we repeat the arguments presented in the proof of
Proposition \ref{s09}. In the present context, $\tau_1$ and
$\eta_{\tau_1}$ are independent by the Markov property, and
$\eta_{H(\mc E_{\mb x})} = \eta^{\mb x}$. We may therefore
skip the coupling arguments of Proposition \ref{s09}. 

In contrast, we need to show that
\begin{equation}
\label{16}
\lim_{A\to\infty} \lim_{\beta\to\infty}  \max_{\zeta \in \mc N (\eta^{\mb x})} 
\mb P^\beta_{\zeta} \big[ H({\mc F}) > A e^\beta \big] \;=\; 0\;.
\end{equation}
Starting from $\zeta\in \mc N(\eta^{\mb x})$, in a time of order one
the process reaches $\mc E$. It follows from Proposition \ref{s16}
that once at $\mc E$ in a time of order $e^\beta$ the process reaches
one of the absorbing point $\{\eta^{\mb x} : \mb x\in \Lambda_L\}$ of
the asymptotic Markovian dynamics characterized by the rates
$r(\,\cdot\,,\,\cdot\,)$. This proves \eqref{16}.

It follows from this result and the proof of Proposition \ref{s09}
that to prove the second assertion of the proposition it is enough to
show that
\begin{equation*}
  \lim_{\beta\to\infty} \sum_{\zeta\in \mc N(\eta^{\mb x})} \mb P^\beta_{\zeta} 
  \big[ H({\mc F}) \not = H(\mc E_{\mb x}) \big] \;=\; Z\;.
\end{equation*}
Since $H({\mc E}) \le H({\mc F}) \le  H(\mc E_{\mb x})$, by the
strong Markov property we may rewrite the previous probability as
\begin{equation*}
\mb E^\beta_{\zeta} \Big[ \mb P^\beta_{\eta(H(\mc E))} 
\big[ H({\mc F}) \not = H(\mc E_{\mb x}) \big]\, \Big]\;.
\end{equation*}
We computed in Lemma \ref{bs01} the asymptotic distribution of
$\eta(H(\mc E))$ and we represented by $q(\mc E_j, \mc E_{\mb y})$ the
probability that the asymptotic process starting from a set $\mc E_j$,
$1\le j\le \kappa$, eventually reaches the absorbing state $\mc
E_{\mb y}$, $\mb y\in \Lambda_L$. The second assertion of the
proposition follows from these two results.

We now turn to the third assertion of the proposition. Fix $\mb y \not =
\mb x$. This argument is also similar to the one of Proposition
\ref{s09}. Denote by $\{H_j : j\ge 1\}$ the successive return times to
$\mc F$:
\begin{equation*}
H_1 = H^+(\mc F)\;, \quad H_{j+1} = H^+(\mc F) \circ \theta_{H_j}\;,
\quad j\ge 1\;.
\end{equation*}
With this notation,
\begin{equation}
\label{13}
\mb P^\beta_{\eta^{\mb x}} 
\big[ \eta(H(\breve{\mc F}_{\mb x})) = \eta^{\mb y} \big] 
\;=\; \sum_{j\ge 1} \mb P^\beta_{\eta^{\mb x}} 
\big[ \eta(H_k) = \eta^{\mb x} \;, 1\le k\le j-1\;,
\eta(H_j) = \eta^{\mb y} \big]\;.
\end{equation}

By the strong Markov property, if $\tau_1$ stands for the time of the
first jump, for any $\mb z\in \Lambda_L$,
\begin{equation*}
\mb P^\beta_{\eta^{\mb x}} \big[ \eta(H_1) = \eta^{\mb z} \big] \; =\;
\mb E^\beta_{\eta^{\mb x}} \Big[ \mb E^\beta_{\eta_{\tau_1}} \Big[
\mb P^\beta_{\eta_{H(\mc E)}} 
\big[\eta(H_{\mc F}) = \eta^{\mb z} \big]\, \Big]\, \Big]\;.
\end{equation*}
As $\beta\uparrow\infty$, this expression converges to
\begin{equation*}
\frac 18 \, \sum_{\xi\in \mc N(\eta^{\mb x})} 
\sum_{j=1}^\kappa \bb M (\xi, \mc E_j) 
\, q(\mc E_j, {\mc E}_{\mb z})\;.
\end{equation*}
The third assertion of the proposition follows from \eqref{13}, this
identity and the strong Markov property.
\end{proof}

It follows from (1) and (2) that the triple 
\begin{equation}
\label{14}
\text{$(\mc E_{\mb x}, \mc
E_{\mb x}\cup \Delta_0, \eta^{\mb x})$ is in fact a valley of depth
$e^{2\beta}/Z$}\; .
\end{equation}

\begin{corollary}
\label{s20}
The sequence of Markov processes $\{\eta^\beta_t : t\ge 0\}$ exhibits
a tunneling behavior on the time-scale $e^{2\beta}$, with metastable states
$\{\mc E_{\mb x} : \mb x\in\Lambda_L \}$, metastable points $\{\eta^{\mb
  x}\}$ and asymptotic Markov dynamics characterized by the rates
\begin{equation*}
r(\mc E_{\mb x}, {\mc E}_{\mb y}) \;=\; Z \,
\bb Q({\mb x}, {\mb y})\;, \quad
\mb x \not = \mb y \in \Lambda_L \;.    
\end{equation*}
\end{corollary}

\begin{proof}
The proof is similar to the one of Proposition \ref{s16}.  We first
check that assumptions (H0) and (H1) of \cite[Theorem 2.7]{bl2} are
fulfilled. Hypothesis (H1) is trivially satisfied since the sets $\mc
E_{\mb x}$ are singletons.


To prove assumption (H0), denote $\{\eta_t^{\mc F}:t\ge 0\}$ the trace
of the process $\eta^\beta_t$ on $\mc F$, and by $R^{\mc F}_\beta$ the
jump rates of the trace process. Note that in this case of singleton
valleys, the average rates coincide with the rates.  We claim that
$e^{2\beta} R^{\mc F}_\beta(\mc E_{\mb x}, \mc E_{\mb y})$, $\mb x\not
= \mb y\in\Lambda_L$, converges to a limit denoted by $R(\mb x, \mb
y)$.

We may rewrite $e^{2\beta} R^{\mc F}_\beta(\mc E_{\mb x}, \mc E_{\mb
  y})$ as $e^{2\beta} R^{\mc F}_\beta(\mc E_{\mb x}, \check{\mc
  F}_{\mb x}) \times [R^{\mc F}_\beta(\mc E_{\mb x}, \mc E_{\mb y})/
R^{\mc F}_\beta(\mc E_{\mb x}, \check{\mc F}_{\mb x})]$. By
\cite[Corollary 4.4]{bl4}, $R^{\mc F}_\beta(\mc E_{\mb x}, \mc E_{\mb
  y})/ R^{\mc F}_\beta(\mc E_{\mb x}, \check{\mc F}_{\mb x})$
converges to a number $p(\mc E_{\mb x}, \mc E_{\mb y})\in [0,1]$.  On
the other hand, by \cite[Lemma 6.7]{bl2}, $e^{2\beta} R^{\mc
  F}_\beta(\mc E_{\mb x}, \check{\mc F}_{\mb x}) = e^{2\beta}
\Cap_K (\mc E_{\mb x}, \check{\mc F}_{\mb x})/\mu_K (\mc
E_{\mb x})$. Clearly, $G_K(\mc E_{\mb x}, \check{\mc F}_{\mb x}) =
e^{-2\beta} \mu_K(\eta^{\mb x})$. Hence, assumption (H0) follows
from \eqref{17}.

In view of \cite[Lemma 10.2]{bl4}, Proposition \ref{s02} and
\eqref{14}, $e^{2\beta}$ $R^{\mc F}_\beta(\mc E_{\mb x}, \mc E_{\mb
  y})$, $\mb x\not = \mb y\in\Lambda_L$, converges to $Z \bb Q(\mb x ,
\mb y)$.

It remains to show property (M3) of tunneling, which states that the
time spent outside $\mc F$ is negligible. Fix $\mb
x\in\Lambda_L$. Denote by $\{H_j : j\ge 1\}$ the times of the
successive returns to $\mc F$: $H_1 = H^+(\mc F)$, $H_{j+1} = H^+(\mc
F) \circ \theta_{H_j}$, $j\ge 1$. To prove (M3), it is enough to show
that for all $t>0$
\begin{equation}
\label{15}
\begin{split}
& \lim_{k\to\infty} \lim_{\beta\to\infty} \mb P^\beta_{\eta^{\mb x}}
\big[ H_k \le t e^{2\beta}\big] \;=\;0 \quad\text{ and}\\
&\quad \lim_{\beta\to\infty} \mb E^\beta_{\eta^{\mb x}} 
\Big[ e^{-2\beta} \int_0^{H_k \wedge t e^{2\beta}} 
\mb 1\{\eta^\beta_{s} \in \Delta_0\} \, ds \Big] \;=\;0
\quad \text{for all $k\ge 1$.}
\end{split}
\end{equation}

Since $H_1= H^+(\mc F)$ is greater than the time of the first jump,
$H_1$ is bounded below by an exponential time of parameter $8
e^{-2\beta}$, $\mb P^\beta_\eta$ almost surely for all $\eta\in \mc
E$. The first line of \eqref{15} follows from this observation and
from the strong Markov property.

To estimate the second term of of \eqref{15}, fix $k\ge 1$ and rewrite
the time integral as $\sum_{0\le j<k} \int_{H_j \wedge t
  e^{2\beta}}^{H_{j+1} \wedge t e^{2\beta}}$. For a fixed $j$, the
integral vanishes unless $H_j < t e^{2\beta}$. Hence, by the strong
Markov property, the expectation is less than or equal to
\begin{equation*}
k\, \max_{\mb y\in\Lambda_L} \mb E^\beta_{\eta^{\mb y}} 
\Big[ e^{-2\beta} \int_0^{H_1  \wedge t e^{2\beta}} 
\mb 1\{\eta^\beta_{s} \in \Delta_0\} \, ds \Big] \;.
\end{equation*}
Recall that we denoted by $\mc F(\eta^{\mb y})$ the set of
configurations which can be reached from $\eta^{\mb y}$ by a jump of
rate $e^{-2\beta}$. By the strong Markov property, this expression is
bounded by
\begin{equation*}
k\, \max_{\mb y\in\Lambda_L} \max_{\xi\in \mc F(\eta^{\mb y})} 
\mb E^\beta_{\xi} \big[ e^{-2\beta} H(\mc F)  \wedge t  \big]\;. 
\end{equation*}
By \eqref{16} this expression vanishes as $\beta\uparrow\infty$.  
\end{proof}

\section{General results}

We prove in this section an useful general result.  Fix a sequence
$(E_N: N\ge 1)$ of countable state spaces. The elements of $E_N$ are
denoted by the Greek letters $\eta$, $\xi$. For each $N\ge 1$ consider
a matrix $R_N : E_N \times E_N \to \bb R$ such that $R_N(\eta, \xi)
\ge 0$ for $\eta \not = \xi$, $-\infty < R_N (\eta, \eta)\le 0$ and
$\sum_{\xi\in E_N} R_N(\eta,\xi)=0$ for all $\eta\in E_N$.

Let $\{\eta^N_t : t\ge 0\}$ be the {\sl minimal} right-continuous
Markov process associated to the jump rates $R_N(\eta,\xi)$ \cite{n}.
It is well known that $\{\eta^N_t : t\ge 0\}$ is a strong Markov
process with respect to the filtration $\{\mc F^N_t : t\ge 0\}$ given
by $\mc F^N_t = \sigma (\eta^N_s : s\le t)$. Let $\mb P_{\eta}$,
$\eta\in E_N$, be the probability measure on $D(\bb R_+,E_N)$ induced
by the Markov process $\{\eta^N_t : t\ge 0\}$ starting from $\eta$.

Consider two sequences $\ms W = (W_N\subseteq E_N : N\ge 1)$, $\ms B =
(B_N \subseteq E_N : N\ge 1)$ of subsets of $E_N$, the second one
containing the first and being properly contained in $E_N$: $W_N
\subseteq B_N \subsetneqq E_N$.  Fix a point $\bs \xi = (\xi_N \in W_N
: N\ge 1)$ in $\ms W$ and a sequence of positive numbers $\bs \theta =
(\theta_N : N\ge 1)$.

Next result states an obvious fact. We may add to the basin $\ms B$ of
a valley $(\ms W, \ms B, \bs \xi)$ a set $\ms C$ never visited by the
process without modifying the properties of the valley.

\renewcommand{\theenumi}{\arabic{enumi}}
\renewcommand{\labelenumi}{(\theenumi)}

\begin{lemma}
\label{s18}
Assume that the triple $(\ms W, \ms B, \bs \xi)$ is a valley of depth
$\bs \theta$ and attractor $\bs \xi$. Let $\ms C = (C_N \subset E_N :
N\ge 1)$ be a sequence of sets such that $B_N^c$ is attained before
$C_N$ when starting from $W_N$:
\begin{equation}
\label{12}
\lim_{N\to\infty} \inf_{\eta\in W_N} 
\mb P_{\eta} \big[ H_{B_N^c} < H_{C_N} \big] \;=\; 1\; .
\end{equation}
Then, the triple $(\ms W, \ms B \cup \ms C, \bs \xi)$ is a
valley of depth $\bs \theta$ and attractor $\bs \xi$.
\end{lemma}

\begin{proof}
We have to check the three conditions of \cite[Definition
2.1]{bl2}. The first one is obvious because $B_N^c \supset (B_N\cup
C_N)^c$. On the event $\{H_{B_N^c} < H_{C_N}\}$, $H_{B_N^c} =
H_{(B_N\cup C_N)^c}$. Hence, the convergence in distribution of
$H_{(B_N\cup C_N)^c}/\theta_N$ to a mean one exponential variable
follows from \eqref{12} and from the one of $H_{B_N^c}/\theta_N$. For
the same reasons, on the set $\{H_{B_N^c} < H_{C_N}\}$,
$\int_0^{H_{B_N^c}} \mb 1\{\eta^N_s \in A\}\, ds = \int_0^{H_{(B_N\cup
    C_N)^c}} \mb 1\{\eta^N_s \in A\}\, ds$. In particular, property
(V3) for the triple $(\ms W, \ms B \cup \ms C, \bs \xi)$ follows from
\eqref{12} and (V3) for the valley $(\ms W, \ms B, \bs \xi)$.
\end{proof}

\medskip{\bf Acknowledgments.} The authors would like to thank
B. Gois, R. Misturini and one of the referees for their careful
reading and their suggestions which permitted to clarify the text.

\end{document}

%% file: fig1a.pdf_tex
\begingroup%
  \makeatletter%
  \providecommand\color[2][]{%
    \errmessage{(Inkscape) Color is used for the text in Inkscape, but the package 'color.sty' is not loaded}%
    \renewcommand\color[2][]{}%
  }%
  \providecommand\transparent[1]{%
    \errmessage{(Inkscape) Transparency is used (non-zero) for the text in Inkscape, but the package 'transparent.sty' is not loaded}%
    \renewcommand\transparent[1]{}%
  }%
  \providecommand\rotatebox[2]{#2}%
  \ifx\svgwidth\undefined%
    \setlength{\unitlength}{963.77949219bp}%
    \ifx\svgscale\undefined%
      \relax%
    \else%
      \setlength{\unitlength}{\unitlength * \real{\svgscale}}%
    \fi%
  \else%
    \setlength{\unitlength}{\svgwidth}%
  \fi%
  \global\let\svgwidth\undefined%
  \global\let\svgscale\undefined%
  \makeatother%
  \begin{picture}(1,0.23529413)%
    \put(0,0){\includegraphics[width=\unitlength]{fig1a.pdf}}%
    \put(0.14117647,0.02352941){\color[rgb]{0,0,0}\makebox(0,0)[lb]{\smash{$\Omega^1$}}}%
    \put(0.25882355,0.02352941){\color[rgb]{0,0,0}\makebox(0,0)[lb]{\smash{$\Omega^2$}}}%
    \put(0.3764706,0.02352941){\color[rgb]{0,0,0}\makebox(0,0)[lb]{\smash{$\Omega^3$}}}%
    \put(0.49411765,0.02352941){\color[rgb]{0,0,0}\makebox(0,0)[lb]{\smash{$\Omega^4$}}}%
    \put(0.61764708,0.02352941){\color[rgb]{0,0,0}\makebox(0,0)[lb]{\smash{$\Omega^3$}}}%
    \put(0.73529415,0.02352941){\color[rgb]{0,0,0}\makebox(0,0)[lb]{\smash{$\Omega^2$}}}%
    \put(0.05882353,0.00000001){\color[rgb]{0,0,0}\makebox(0,0)[lb]{\smash{$\Omega^0$}}}%
    \put(0.19411765,0.0529412){\color[rgb]{0,0,0}\makebox(0,0)[lb]{\smash{$\Gamma_2$}}}%
    \put(0.33529413,0.0529412){\color[rgb]{0,0,0}\makebox(0,0)[lb]{\smash{$\Gamma_3$}}}%
    \put(0.42941176,0.0529412){\color[rgb]{0,0,0}\makebox(0,0)[lb]{\smash{$\Gamma_4$}}}%
    \put(0.57058824,0.0529412){\color[rgb]{0,0,0}\makebox(0,0)[lb]{\smash{$\Gamma_4$}}}%
    \put(0.66470592,0.0529412){\color[rgb]{0,0,0}\makebox(0,0)[lb]{\smash{$\Gamma_3$}}}%
    \put(0.29670063,0.12501927){\color[rgb]{0,0,0}\makebox(0,0)[lb]{\smash{$\Delta_2$}}}%
    \put(0.80588239,0.0529412){\color[rgb]{0,0,0}\makebox(0,0)[lb]{\smash{$\Gamma_2$}}}%
    \put(0.85294123,0.02352941){\color[rgb]{0,0,0}\makebox(0,0)[lb]{\smash{$\Omega^1$}}}%
    \put(0.90000002,0.0529412){\color[rgb]{0,0,0}\makebox(0,0)[lb]{\smash{$\Gamma_1$}}}%
    \put(0.94705881,0.00000001){\color[rgb]{0,0,0}\makebox(0,0)[lb]{\smash{$\Omega^0$}}}%
    \put(0.49411765,0.15294118){\color[rgb]{0,0,0}\makebox(0,0)[lb]{\smash{$\Lambda$}}}%
    \put(0.02352941,0.0529412){\color[rgb]{0,0,0}\makebox(0,0)[lb]{\smash{$\Gamma_1$}}}%
    \put(0.10000001,0.0529412){\color[rgb]{0,0,0}\makebox(0,0)[lb]{\smash{$\Gamma_1$}}}%
  \end{picture}%
\endgroup%